\documentclass[final,3p,times]{article}
\usepackage{amssymb, amsmath, amsthm, mathtools, booktabs, caption, multirow, braket, pgfplots, bm, textcomp}
\pgfplotsset{/pgf/number format/use comma,compat=newest}

\usepackage{etex}
\usepackage{epstopdf}
\usepackage[all]{xy}

\usepackage{siunitx}
\usepackage[normalem]{ulem}
\usepackage{graphicx}

\usepackage[labelformat=simple]{subcaption}
\graphicspath{{./Figures/}}

\newtheorem{theorem}{Theorem}[section]
\newtheorem{lemma}[theorem]{Lemma}
\newtheorem{proposition}[theorem]{Proposition}

\newtheorem{definition}[theorem]{Definition}
\newtheorem{example}[theorem]{Example}

\newtheorem{remark}[theorem]{Remark}

\numberwithin{equation}{section}

\newcommand\NN{\mathbb{{N}}}
\newcommand\RR{\mathbb{{R}}}
\newcommand\CC{\mathbb{{C}}}
\newcommand\ZZ{\mathbb{{Z}}}

\DeclarePairedDelimiter{\abs}{\lvert}{\rvert}
\DeclarePairedDelimiter{\norma}{\lVert}{\rVert}

\newcommand{\smoother}[3]{
V_{#1, \text{#2}}^{#3}
}

\title{Multigrid methods: grid transfer operators and subdivision schemes}
\author{M. Charina, M. Donatelli, L. Romani and V. Turati}
\date{}

\newcommand{\Addresses}{{
\footnotesize

M. Charina, \textsc{Fakult\"at f\"ur Mathematik}, \emph{Universit\"at Wien}, Oskar-Morgenstern-Platz 1, 1090 Wien (Austria), \textit{maria.charina@univie.ac.at}

\smallskip 
 
M. Donatelli, \textsc{Dipartimento di Scienza e Alta Tecnologia}, \emph{Universit\`a dell'Insubria}, via Valleggio 11, 22100 Como (Italia), \textit{marco.donatelli@uninsubria.it}

\smallskip
  
L. Romani, \textsc{Dipartimento di Matematica e Applicazioni}, \emph{Universit\`a di Milano-Bicocca}, via Cozzi 53, 20125 Milano (Italia), \textit{lucia.romani@unimib.it}

\smallskip

V. Turati, \textsc{Dipartimento di Scienza e Alta Tecnologia}, \emph{Universit\`a dell'Insubria}, via Valleggio 11, 22100 Como (Italia), \textit{vturati@studenti.uninsubria.it}

}}

\begin{document}

\maketitle

\begin{abstract} 
The convergence rate of a multigrid method depends on the properties of the smoother and the so-called grid 
transfer operator. In this paper we define and analyze new grid transfer operators with a generic 
cutting size which are applicable for high order problems. We enlarge the class of available geometric grid 
transfer operators by relating the symbol analysis of the coarse grid correction with the approximation 
properties of univariate subdivision schemes.  We show that the polynomial generation property and stability 
of a subdivision scheme are crucial for convergence and optimality of the corresponding multigrid method. 
We construct a new class of grid transfer operators from primal binary and ternary pseudo-spline symbols. 
Our numerical results illustrate the behavior of the new grid transfer operators. 
\end{abstract}

\Addresses


\section{Introduction}
Multigrid methods are used for solving linear systems of equations 
\begin{equation}\label{eq:sys}
A_n  \mathbf{x} = \mathbf{b}_n, \quad \mathbf{x} \in \CC^n,
\end{equation} 
with symmetric and positive definite system matrices $A_n \in \CC^{n \times n}$ and $\mathbf{b}_n \in \CC^n$, $n \in \NN$.
A basic two-grid method combines the action of a smoother and a coarse grid correction: the smoother is often a simple iterative method such as Gauss-Seidel; the coarse grid correction amounts to solving the residual equation exactly on a coarser grid. A V-cycle multigrid method solves the residual equation approximately within the
recursive application of the two-grid method, until the coarsest level is reached and there the resulting small 
system of equations is solved exactly \cite{Brandt,mgm-book}.

The choice of the grid transfer operator is crucial for the definition of an effective multigrid method and
becomes cumbersome especially for high order problems or on complex domains. Several algebraic multigrid 
methods have been designed to overcome these difficulties \cite{Notay, RS1987, Vanek}. Simple geometric grid transfer operators are of interest for both geometric and algebraic multigrid methods due to their simplicity and 
applicability for re-discretizations of the problem at coarser levels. The common choice for the grid transfer 
operators are interpolation operators \cite{mgm-book}. We show that a variety of new interpolating and approximating schemes, developed to design curves and surfaces via subdivision, can be successfully used as grid transfer operators 
for multigrid methods.

The main contribution of this paper are the sufficient conditions on the symbol of a univariate subdivision scheme
of general arity that guarantee that the corresponding grid transfer operator leads to an optimal multigrid method. 
A stationary iterative method is called optimal whenever its convergence rate is linear and the computational cost of each iteration is proportional to the cost of a matrix vector product. To the best of our knowledge, a hint on the possible link between multigrid methods and subdivision schemes can be only found in \cite{fluid}. 
However, \cite{fluid} only presents a special multigrid method with a structure similar to cascadic multigrid 
\cite{cascadic} without any theoretical analysis.

To clarify the link between multigrid and subdivision, we start by recalling that local Fourier analysis (LFA) \cite{Brandt} is a classical 
tool for the convergence analysis of multigrid methods with applications in partial differential equations (PDEs).
In \cite{NLL2010}, it has been proved that the multigrid analysis based on the symbol of circulant or tau matrix
\footnote{The circulant and tau matrices are algebra of matrices diagonalized by the Fourier and the sine matrix, respectively, such that the eigenvalues are a uniform sampling, in a reference interval, of a specific function called symbol.} is an extension of the LFA to linear systems arising from problems that are not necessarily defined by PDEs. 
In the case of PDEs, the symbol that is analyzed in \cite{NLL2010}, and proposed for the first time in \cite{FS} for tau matrices, does not depend on the discretization step and the information on the order of the problem 
is retrieved from the order of the zero of the symbol (see \cite{NLL2010} for details). Similarly, the properties of a subdivision scheme are encoded into the associated Laurent polynomial, analogue of the symbol of a circulant matrix. Due to this analogy, in this paper, we recast the optimality conditions on the grid transfer operator of a multigrid method in terms of certain approximation properties of the corresponding subdivision scheme. In particular,
we slightly relax usual sufficient conditions for convergence and optimality of two-grid methods and translate them into subdivision language (we require polynomial generation property).
For the V-cycle method, for a cutting size larger than two, we first derive sufficient conditions that imply the approximation condition introduced in \cite{RS1987}. Our result 
generalizes the two-grid analysis in \cite{BIT2011}. 
These new sufficient conditions allow us then to obtain the optimality of the V-cycle method under the assumption of $\ell^{\infty}$-stability or under the Cohen's condition on the subdivision symbol. We apply our results to binary and ternary primal pseudo-splines and show that  their symbols define appropriate grid transfer operators which satisfy the above mentioned optimality conditions. Finally, our numerical experiments show the effectiveness of the new grid transfer operators based on binary or ternary pseudo-splines. We apply the corresponding multigrid methods for solving high order problems and for solving linear systems derived via isogeometric approach 
\cite{Don_Gar_Mann_Serra_Spel2,Don_Gar_Mann_Serra_Spel}.

Even if the theoretical analysis of multigrid is done in the case of circulant matrices, the resulting multigrid methods are applicable for solving more general linear systems of equations, in particular, those with Toeplitz system matrices. Indeed, it is well-known that Toeplitz matrices are also defined by means of symbols and that they are well approximated by circulant matrices \cite{Bottcher,ToepReview}. 

In order to keep the presentation simple, in this paper, we discuss only one dimensional problems and univariate subdivision schemes.  This allows for the first, transparent and straightforward exposition of the links between the symbol analysis for multigrid methods and symbols of subdivision schemes. Our results can be extended in many directions. In particular, the study of multivariate anisotropic problems as well as of problems of fluid dynamics, connection between multigrid and dual subdivision schemes - these are of future interest.

The paper is organized as follows.  In section \ref{sec:notation}, we recall basic properties of univariate 
subdivision schemes and multigrid methods. The emphasis is on the multigrid convergence analysis  based 
on the study of symbols of certain circulant matrices. In section \ref{sec:conv}, we provide new sufficient conditions  for convergence and optimality of multigrid methods with a cutting size larger than two. There, we also investigate the properties of subdivision schemes of general arity that are crucial for defining grid transfer operators of optimal multigrid methods.
In section \ref{sec:splines}, we recall the definition and properties of primal binary and ternary pseudo-splines and show that 
their symbols satisfy conditions stated in section \ref{sec:conv}. In section \ref{sec:num_examples},
we illustrate the properties of the new grid transfer operators based on primal pseudo-splines with
numerical experiments. We summarize our results and outline possible future research directions in Section \ref{sec:concl}.

\section{Background and notation}\label{sec:notation}

In this section, in subsection \ref{sec:subdivision}, we introduce the basic facts about univariate subdivision schemes. Then in
subsection \ref{sec:multigrid_generic}, we give a short overview of basic facts about multigrid methods that already hint at a possible
link between multigrid and subdivision, see Remark \ref{rem:connection_to_subdivision}.

\subsection{Univariate subdivision} \label{sec:subdivision}

Let $g \in \NN, \, g \geq 2$ and $\mathbf{p} = \{ \mathrm{p}_{\alpha} \in \RR \, : \, \alpha \in \ZZ \} \in \ell_0 (\ZZ)$ a finite sequence of real numbers.
The arity $g$ and the \emph{mask} $\mathbf{p}$ are used to define the subdivision operator $S_{\mathbf{p}} \colon \ell(\ZZ) \to \ell(\ZZ)$, which is a linear operator such that
\[
\left ( S_{\mathbf{p}} \mathbf{c} \right )_{\alpha} := \sum_{\beta \in \ZZ} \mathrm{p}_{\alpha - g \beta} \mathrm{c}_{\beta}, \quad \alpha \in \ZZ, \quad \forall \, \mathbf{c} \in \ell (\ZZ).
\]
A subdivision scheme $S_{\mathbf{p}}$ of arity $g$ associated with the mask $\mathbf{p}$ is the recursive application of the subdivision operator $S_{\mathbf{p}}$ to an initial data sequence
$\mathbf{c}^{(0)} = \{ \mathrm{c}_{\alpha}^{(0)} \in \RR \, : \, \alpha \in \ZZ \} \in \ell(\ZZ)$, namely
\begin{equation}
 \label{eq:subdivision}
 \mathbf{c}^{(k+1)} := S_{\mathbf{p}} \mathbf{c}^{(k)}, \quad k \in \NN_0.
\end{equation}
Notice that $\mathbf{c}^{(k+1)} = S_{\mathbf{p}} \mathbf{c}^{(k)} = \cdots = (S_{\mathbf{p}})^{k+1} \mathbf{c}^{(0)}$.

We recall that $\ell^{\infty}(\ZZ) \subset \ell(\ZZ)$ is the Banach space of bounded sequences $\mathbf{c}$ with the norm
\[
 \norma{\mathbf{c}}_{\infty} := \sup_{\alpha \in \ZZ} \, \abs{\mathrm{c}_{\alpha}}.
\]

\begin{definition}
A subdivision scheme $S_{\mathbf{p}}$ is \emph{convergent} if for any initial data $\mathbf{c} \in \ell^{\infty} (\ZZ)$ there exists a uniformly continuous function
$F_{\mathbf{c}} \in C(\RR)$ such that
\[
\lim_{k \to \infty} \, \sup_{\alpha \in \ZZ} \quad \abs*{\, F_{\mathbf{c}} \left ( \frac{\alpha}{g^k} \right ) - \left ( S_{\mathbf{p}}^k \mathbf{c} \right )_{\alpha} \, } = 0.
\]
\end{definition}

The particular choice of the initial data $\bm{\delta} = \Set{\delta_{\alpha,0} \, : \, \alpha \in \ZZ} = \Set{\ldots, 0, 0, 1, 0, 0, \ldots}$
defines the so-called \emph{basic limit function} $\phi = F_{\bm{\delta}}$. Notice that $\phi$ is compactly supported since the mask
$\mathbf{p} \in \ell_0 (\ZZ)$ is a finite sequence. It is well-known that the basic limit function $\phi$ satisfies the refinement equation
\begin{equation} \label{eq:ref_eq}
\phi (t) = \sum_{\alpha \in \ZZ} \mathrm{p}_{\alpha} \phi (gt - \alpha), \quad t \in \RR.
\end{equation}
Thus, due to the linearity of $S_{\mathbf{p}}$, for any initial data $\mathbf{c} \in \ell(\ZZ)$, $ \mathbf{c} =\displaystyle \sum_{\alpha \in \ZZ} \mathrm{c}_\alpha \bm{\delta}(\cdot - \alpha)$, it holds
\[
F_{\mathbf{c}} = \lim_{k \to \infty} S_p^k \mathbf{c} = \sum_{\alpha \in \ZZ} \mathrm{c}_{\alpha} \phi(\cdot - \alpha).
\]
For more details on subdivision, see the seminal work of Cavaretta et al. \cite{Cava_Dam_Micc} and the survey by Dyn and Levin \cite{Dyn_Lev}.

The Laurent polynomial
\[
p(z) = \sum_{\alpha \in \ZZ} \mathrm{p}_{\alpha} z^{\alpha}, \qquad z \in \CC \setminus \set{0},
\]
is the \emph{symbol} of the subdivision scheme $S_{\mathbf{p}}$. To establish the link between grid transfer operators and
subdivision schemes, we only consider subdivision schemes whose mask
$\mathbf{p}$ is odd symmetric, namely $\mathrm{p}_{-\alpha} = \mathrm{p}_{\alpha}, \, \alpha \in \ZZ$. In terms of symbols, it reads $p(z) = p(z^{-1})$.
Thus, the associated symbol is of the form
\begin{equation} \label{eq:symbol_trigo_poly}
p(z) = \mathrm{p}_0 + \sum_{\alpha \in \NN} \mathrm{p}_{\alpha} \left (z^{-\alpha} + z^{\alpha} \right), \qquad z \in \CC \setminus \set{0}.
\end{equation}
Notice that for $z = e^{- \mathrm{i} x}, \, x \in \RR$, the symbol $p(z)$ becomes a real-valued trigonometric polynomial.

We next define some properties of subdivision schemes which we use for the analysis of convergence and optimality of multigrid methods.
We first introduce the concept of $\ell^{\infty}$-stability.

\begin{definition}
Let $f \in L^{\infty} (\RR)$ be compactly supported. We say that $f$ is $\ell^{\infty}$-stable  if there exist constants $0 < A \leq B  < \infty$ such that
\begin{equation} \label{eq:def_stability}
A \norma*{\mathbf{c}}_{\infty} \leq \norma*{\sum_{\alpha \in \ZZ} \mathrm{c}_{\alpha} f(\cdot - \alpha) }_{L^{\infty} ( \RR )} \leq B \norma*{\mathbf{c}}_{\infty}, \qquad \forall \, \mathbf{c} \in \ell^{\infty} ( \ZZ).
\end{equation}
\end{definition}
We also define the polynomial generation property of $S_{\mathbf{p}}$. This property plays a fundamental role in our analysis of multigrid.
We denote by $\Pi_d$ the space of polynomials of degree $d \in \NN_0$.

\begin{definition} \label{d:poly_generation}
A convergent subdivision scheme $S_{\mathbf{p}}$ \emph{generates} polynomials up to degree $d$ if
\[
\forall \, \mathbf{c} = \Set{ \mathrm{c}_{\alpha} = \pi(\alpha) \, : \, \alpha \in \ZZ}, \quad \, \pi \in \Pi_{d}, \quad \sum_{\alpha \in \ZZ} \mathrm{c}_\alpha \phi(\cdot - \alpha) = \tilde{\pi} \in \Pi_{d}.
\]
\end{definition}
The property of polynomial generation has been studied, e.g., by Cavaretta et al. in \cite{Cava_Dam_Micc} or by Levin in \cite{Levin_gen_non_unif}. One of the results in \cite{Levin_gen_non_unif}
states that the limit polynomial $\tilde{\pi} \in \Pi_{d}$ has the same leading coefficient as $\pi \in \Pi_{d}$.
Cavaretta et al. also showed that for a binary ($g=2$) subdivision scheme polynomial generation is characterized in terms of its symbol.
Their result is generalized for arbitrary arity $g$, e.g., in \cite{Levin_gen_non_unif}.

We denote the set of the $g$-th roots of unity by 
\[
E_g := \Set{e^{- \mathrm{i} \frac{2 \pi j}{g}} \, : \, j=0, \dots, g-1}.
\]

\begin{theorem}[\cite{Cava_Dam_Micc, Levin_gen_non_unif}]\label{t:poly_generation}
A convergent subdivision scheme $S_{\mathbf{p}}$ generates polynomials up to degree $d$ if and only if
\begin{equation} \label{conditions:poly_generation}
 D^j p (\varepsilon) = 0 \quad \text{for} \quad \varepsilon \in E_g \setminus \set{1}, \quad j = 0, \dots, d.
\end{equation}
\end{theorem}

Thus, polynomial generation is guaranteed by the fact that the symbol $p(z)$ and its derivatives vanish at all the $g$-th roots of unity except at 1.
Finally, the condition in \eqref{conditions:poly_generation} is equivalent to requiring  that the symbol $p(z)$ has the following structure
\begin{equation} \label{eq:symbol_factorization_generation}
p(z) = \left (1+z+z^2+\dots+z^{g-1} \right )^{d + 1} \, b(z),
\end{equation}
for some Laurent polynomial $b(z)$. One also requires that $b(1) = g^{-d}$, i.e. $p(1)=g$. The reason for that is that the zero conditions of order $d+1$
$$
p(1) = g \qquad \text{and} \qquad  D^j p (\varepsilon) = 0 \quad \text{for} \quad \varepsilon \in E_g \setminus \set{1}, \quad j = 0, \dots, d,
$$
are necessary for convergence and regularity of subdivision, see e.g. \cite{Cava_Dam_Micc}.


\subsection{Multigrid methods for symmetric positive definite system matrices} \label{sec:multigrid_generic}

\vspace{0.3cm} Let $n \in \NN$ be a positive integer. To design two-grid and multigrid iterative methods for solving linear systems of the form \eqref{eq:sys}, we define 
\begin{itemize}
\item $N \in \NN$, $N < n$ the dimension of coarse space at which we project our problem,
\item the grid transfer operator $P_n \in \CC^{n \times N}$, $\hbox{rank}(P_n)=N$, and
\item a class ${\cal V}(\cdot)$ of iterative methods of the form
\begin{equation} \label{eq:iter_methods}
\begin{split}
{\cal V} ( \mathbf{x}_n^{(k)} ) := \mathbf{x}_n^{(k+1)} = V_n \mathbf{x}_n^{(k)} + \mathbf{\tilde{b}}_n, \qquad k \in \NN_0, \\
V_n = I_n - W_n^{-1} A_n, \qquad \mathbf{\tilde{b}}_n = W_n^{-1} \mathbf{b}_n \in \CC^n,
\end{split}
\end{equation}
where $W_n$ is a nonsingular matrix such that $W_n=A_n-B_n$ for some $B_n \in \CC^{n \times n}$.
\end{itemize}

The multigrid methods considered in this paper are based on the Galerkin approach defined by the two following conditions:
\begin{itemize}
\item the restriction is the conjugate transpose of the prolongation, i.e., $P_n^H$,
\item the coarser matrix is defined by $A_N = P_n^H A_n P_n$. 
\end{itemize}

The Galerkin approach is crucial for our theoretical analysis. Nevertheless, the proposed grid transfer operators can 
be effectively applied also in geometric multigrid methods.

\subsubsection{Convergence and optimality of the Two-Grid method}

\vspace{0.3cm} Let $\smoother{n}{pre}{}$ and $\smoother{n}{post}{}$ be some iterative methods from ${\cal V}(\cdot)$ and
$\nu_{\text{pre}}, \nu_{\text{post}} \in \NN_0$. The simplest of multigrid methods is the Two-Grid Method (TGM), whose $k$-th
iteration is defined by the following algorithm
\begin{equation} \label{alg:tgm}
\begin{array}{c}
\text{TGM} (\smoother{n}{pre}{\nu_{\text{pre}}}, \smoother{n}{post}{\nu_{\text{post}}},P_n) (\mathbf{x}_n^{(k)}) \\
\hline \\
\smallskip
\begin{array}{lll}
0. & \mathbf{\tilde{x}}_n = {\cal V}_{n,\text{pre}}^{\nu_{\text{pre}}}(\mathbf{x}_n^{(k)}) &  \, (\text{pre-smoother}) \\
\smallskip
1. & \mathbf{r}_n = \mathbf{b}_n - A_n \mathbf{\tilde{x}}_n \in \CC^n &  \, (\text{residual}) \\
\smallskip
2. & \mathbf{r}_N = P_n^H \mathbf{r}_n \in \CC^N &  \, (\text{restriction of residual to coarser grid}) \\
\smallskip
3. & A_N = P_n^H A_n P_n \in \CC^{N \times N} & \, (\text{restriction of $A_n$ to coarser grid}) \\
\smallskip
4. & \text{Solve} \quad A_N \mathbf{e} = \mathbf{r}_N &  \, (\text{error equation}) \\
\smallskip
5. & \mathbf{\hat{x}}_n = \mathbf{\tilde{x}}_n + P_n \mathbf{e}_N \in \CC^n &  \, (\text{correction of $\mathbf{\tilde{x}}_n$}) \\
\smallskip
6. & \mathbf{x}_n^{(k+1)} = {\cal V}_{n,\text{post}}^{\nu_{\text{post}}} (\mathbf{\hat{x}}_n) &  \, (\text{post-smoother})
\end{array}
\end{array}
\end{equation}

Steps $1.-5.$ in the above algorithm define the \emph{coarse grid correction} (CGC) operator on $\RR^n$ by
\begin{equation} \label{def:TGM_CGC0}
 CGC_n=I_n - P_n \bigl ( P_n^H A_n P_n \bigr )^{-1} P_n^H A_n.
\end{equation}
The global iteration matrix of the TGM is then given by
\begin{equation} \label{eq:tgm_matrix}
TGM = \smoother{n}{post}{\nu_{\text{post}}} \, CGC_n \, \smoother{n}{pre}{\nu_{\text{pre}}}.
\end{equation}
Theorem \ref{t:tgm_optimality} is a well-known result from \cite{RS1987}, which provides sufficient conditions for convergence of TGM.
For simplicity, we state this result in the case $\nu_{\text{pre}} = 0$. If $\nu_{\text{pre}} \neq 0$, then see \cite{RS1987} for the corresponding statement. 
To formulate Theorem \ref{t:tgm_optimality}, we define
\begin{itemize}
\item $D_n \in \CC^{n \times n}$ to be the diagonal matrix with the diagonal entries of $A_n$,
\item the norm $\norma{\cdot}_{A_n} = \norma{A_n^{1/2} \cdot}_2$ generated by the positive definite matrix $A_n$.
\end{itemize}

\begin{theorem} \label{t:tgm_optimality}
Let $A_n \in \CC^{n \times n}$ be positive definite, $\smoother{n}{post}{} \in {\cal V}(\cdot)$ and $P_n \in \CC^{n \times N}$ be an appropriate
grid transfer operator. If
\begin{enumerate}
\item $ \exists \, \alpha > 0$ independent of $n$ such that
\begin{equation} \label{eq:smooth_property}
\norma{\smoother{n}{post}{} \mathbf{x}_n}_{A_n}^2 \leq \norma{\mathbf{x}_n}_{A_n}^2 - \alpha \norma{\mathbf{x}_n}_{A_n D_n^{-1} A_n}^2, \qquad \forall \mathbf{x}_n \in \CC^n,
\end{equation}
\item $\exists \, \gamma > 0$ independent of $n$ such that
\begin{equation} \label{eq:approx_property}
\min_{\mathbf{y} \in \CC^N} \norma{\mathbf{x}_n - P_n \mathbf{y}}_{D_n}^2 \leq \gamma \norma{\mathbf{x}_n}_{A_n}^2, \qquad \forall \mathbf{x}_n \in \CC^n,
\end{equation}
\end{enumerate}
then $\gamma \geq \alpha$ and \[
\norma{\text{TGM}}_{A_n} \leq \sqrt{1 - \alpha /{\gamma} }.
\]
\end{theorem}
Conditions \eqref{eq:smooth_property} and \eqref{eq:approx_property} are called \emph{smoothing} and \emph{approximation} properties, respectively.

\begin{remark}
Note that Theorem \ref{t:tgm_optimality} also implies optimality of TGM, due to $\alpha$ and $\gamma$ being independent of $n$. In other words, the number of iterations needed to reach a given accuracy $\epsilon$ 
is bounded from above by a constant independent of $n$ (but, possibly depending on $\epsilon$).
\end{remark}

\subsubsection{Convergence and optimality of V-cycle}

\vspace{0.3cm} If $N$ is large, then the numerical solution of the linear system at the Step 4. in \eqref{alg:tgm} could be computationally expensive. In this case, one usually applies a multigrid method based on several, 
possibly different, grid transfer operators.
For $\ell \in \NN$, define a decreasing sequence $n = n_0 > n_1 > \dots > n_{\ell-1} > n_{\ell} > 0$ of integers $n_j$.
For each $n_j$, $j = 0, \ldots, \ell-1$, one chooses $P_{n_j} \in \CC^{n_j \times n_{j+1}}$, $\hbox{rank}(P_{n_j})=n_{j+1}$.
Then, for given $\smoother{n_j}{pre}{}$ and $\smoother{n_j}{post}{}$, $j=0, \ldots,\ell-1$, from ${\cal V}(\cdot)$, and for fixed $s \in \NN$, the Multigrid method (MGM)
generates a sequence $\{ \mathbf{x}_n^{(k)} \in \CC^n \, : \, k \in \NN\}$ defined by
\[
\mathbf{x}_n^{(k+1)} = \text{MGM} (\smoother{n}{pre}{\nu_{\text{pre}}}, \smoother{n}{post}{\nu_{\text{post}}}, P_{n_0}, A_n, \mathbf{b}_n, s, 0) (\mathbf{x}_n^{(k)}),
\quad  \nu_{\text{pre}}, \nu_{\text{post}} \in \NN_0,
\]
where the mapping $\hbox{MGM}: \RR^n \rightarrow \RR^n$ is defined iteratively by
\begin{equation} \label{alg:mgm}
\begin{array}{l}
\text{MGM} (\smoother{n_j}{pre}{\nu_{\text{pre}}}, \smoother{n_j}{post}{\nu_{\text{post}}}, P_{n_j }, A_{n_j}, \mathbf{b}_{n_j}, s, j) (\mathbf{x}_{n_j}^{(k)}) \\
\hline \\
\text{If } j=\ell \text{ then solve } A_{n_{\ell}} \mathbf{x}_{n_{\ell}}^{(k+1)} = \mathbf{b}_{n_{\ell}} \\
\text{Else} \\
\begin{array}{ll}
0. & \mathbf{\tilde{x}}_{n_j} = {\cal V}_{{n_j},\text{pre}}^{\nu_{\text{pre}}}(\mathbf{x}_{n_j}^{(k)}) \\
\smallskip
1. & \mathbf{r}_{n_j} =  \mathbf{b}_{n_j} - A_{n_j} \mathbf{\tilde{x}}_{n_j} \in \CC^{n_j} \\
\smallskip
2. & \mathbf{r}_{n_{j+1}} = P_{n_j}^H \mathbf{r}_{n_j} \in \CC^{n_{j+1}} \\
\smallskip
3. & A_{n_{j+1}} = P_{n_j}^H A_{n_j} P_{n_j} \in \CC^{n_{j+1} \times n_{j+1}} \\
\smallskip
4. & \mathbf{x}_{n_{j+1}}^{(k+1)} = 0 \\
\smallskip
    & \text{for } r = 1 \text{ to } s \\
\smallskip
    & \quad \mathbf{x}_{n_{j+1}}^{(k+1)} = \text{MGM} (\smoother{n_{j+1}}{pre}{\nu_{\text{pre}}}, \smoother{n_{j+1}}{post}{\nu_{\text{post}}}, P_{n_{j+1}}, A_{n_{j+1}}, \mathbf{r}_{n_{j+1}}, s, j+1) (\mathbf{x}_{n_{j+1}}^{(k+1)}) \\
\smallskip
5. & \mathbf{\hat{x}}_{n_j} = \mathbf{\tilde{x}}_{n_j} + P_{n_j}  \mathbf{x}_{n_{j+1}}^{(k+1)} \in \CC^{n_{j}} \\
\smallskip
6. & \mathbf{x}_{n_{j}}^{(k+1)} = {\cal V}_{n_j,\text{post}}^{\nu_{\text{post}}} (\mathbf{\hat{x}}_{n_j})
\end{array}
\end{array}
\end{equation}
The choice $s = 1$ corresponds to the well-known V-cycle method \cite{mgm-book}. The iterative structure of V-cycle is
depicted in the following figure.
{\scriptsize
\begin{displaymath}
\xymatrix{
\CC^{n_0} \ar[dr]_{P_{n_0}^H} & & & & & & \CC^{n_0} \\
& \CC^{n_1} \ar[dr]_{P_{n_1}^H} & & & & \CC^{n_1} \ar[ur]_{P_{n_0}} & \\
& & \CC^{n_2} \ar@{.>}[dr]_{P_{n_{\ell-1}}^H} & & \CC^{n_2} \ar[ur]_{P_{n_1}} & & \\
& & & \CC^{n_{\ell}} \ar@{.>}[ur]_{P_{n_{\ell-1}}} & & & }
\end{displaymath}
}
Similarly to the TGM, at each level $j=0, \dots, \ell-1$ of the V-cycle method, one defines the corresponding coarse grid transfer operator by
\begin{equation} \label{def:CGC_nj}
CGC_{n_j} = I_{n_j} - P_{n_j} \bigl ( P_{n_j}^H A_{n_j} P_{n_j} \bigr )^{-1} P_{n_j}^H A_{n_j} \in \CC^{{n_j} \times {n_j}}.
\end{equation}
More precisely, the global iteration matrix of the V-cycle method is $MGM = MGM_0$, where
\begin{equation}
\begin{split}
MGM_{\ell} &= 0 \in \CC^{{n_{\ell}} \times {n_{\ell}}}, \\
MGM_j &= \smoother{n_j}{post}{\nu_{\text{post}}} \, 
\left(I_{n_j} - P_{n_j} \bigl (I_{n_{j+1}} - MGM_{j+1} \bigr )A_{n_{j+1}}^{-1} P_{n_j}^H A_{n_j}\right)
 \, \smoother{n_j}{pre}{\nu_{\text{pre}}}
\end{split}
\end{equation}
for $ j = \ell - 1, \dots, 0$.
The following result is the analogous of Theorem \ref{t:tgm_optimality} for the V-cycle method. We refer to \cite{RS1987} for more details.

\begin{theorem} \label{t:vcycle_optimality}
Let $A_n \in \CC^{n \times n}$ be  positive definite, $\smoother{n_j}{post}{}$, $j=0, \dots, \ell-1$, be from ${\cal V}(\cdot,\cdot)$ and $CGC_{n_j}$, $j=0, \dots, \ell-1$, be from \eqref{def:CGC_nj}. If, for $j=0, \dots, \ell-1$,
\begin{enumerate}
\item $ \exists \, \alpha_j > 0$ independent of $n_j$ such that
\begin{equation} \label{eq:smooth_property_vcycle}
\norma{\smoother{n_j}{post}{} \mathbf{x}_{n_j}}_{A_{n_j}}^2 \leq \norma{\mathbf{x}_{n_j}}_{A_{n_j}}^2 - \alpha_j \norma{\mathbf{x}_{n_j}}_{A_{n_j} D_{n_j}^{-1} A_{n_j}}^2, \qquad \forall \mathbf{x}_{n_j} \in \CC^{n_j},
\end{equation}
\item $\exists \, \gamma_j > 0$ independent of $n_j$ such that
\begin{equation} \label{eq:approx_property_vcycle}
 \norma{CGC_{n_j} \mathbf{x}_{n_j}}_{A_{n_j}}^2 \leq \gamma_j \norma{\mathbf{x}_{n_j}}_{A_{n_j}^2}^2, \qquad \forall \, \mathbf{x}_{n_j} \in \CC^{n_{j}},
\end{equation}
\end{enumerate}
then \[
0< \delta = \min_{j=1,\dots,\ell-1} \frac{\alpha_j}{\gamma_j} <1  \qquad \text{and} \qquad \norma{MGM}_{A_{n}} \leq \sqrt{1-\delta} < 1.
\]
\end{theorem}

Conditions \eqref{eq:smooth_property_vcycle} and \eqref{eq:approx_property_vcycle} are also called \emph{smoothing} and \emph{approximation} properties, respectively. 

\begin{remark}
It is well-known that iterative methods such as Gauss-Seidel, weighted Jacobi and weighted Richardson belong to ${\cal V}(\cdot)$ and satisfy the smoothing property \eqref{eq:smooth_property} or \eqref{eq:smooth_property_vcycle} 
for an appropriate choice of the weights (see e.g.~\cite{V-cycle_opt,lemma_smoother}).  
\end{remark}

Thus, the aim of this paper is to derive simpler sufficient conditions for the approximation properties in \eqref{eq:approx_property} for TGM  and in \eqref{eq:approx_property_vcycle} for the V-cycle method. In the case of circulant system matrices $A_n$,
these sufficient conditions will be given in terms of the properties of subdivision schemes, section \ref{sec:conv}.
We decide to restrict our analysis to the formalism based on circulant matrices, instead of the classical LFA, 
to better clarity the link between the symbol analysis of multigrid and properties of Laurent polynomials used in subdivision.
In the following subsection, we, thus, recall the structure of the multigrid method for circulant system matrices and the related convergence results.

\subsubsection{Algebraic multigrid methods for circulant matrices}

\vspace{0.3cm} We assume that the system matrix $A_n \in \CC^{n \times n}$ in \eqref{eq:sys} is circulant. It is well-known that 
the analysis of multigrid for circulant matrices depicts well the properties of multigrid in the case 
of positive definite Toeplitz system matrices and allows to use the matrix algebra structure.

Let $n=g^k$ with $g \in \NN, \, g \geq 2$ and $k \in \NN$. It is well-known that any circulant matrix $A_n=C_n(f)$ can be defined using the Fourier coefficients
$$
a_j =\frac{1}{2 \pi} \int_{0}^{2\pi} f(x) \, e^{- \mathrm{i} j x} dx, \qquad j = -d, \dots, d,
$$
of the trigonometric polynomial $f \colon [0, 2 \pi ) \to \CC$
\[
 f(x) = \sum_{j=-d}^d a_j \, e^{ \mathrm{i} j x}, \quad x \in [0, 2\pi),
\]
of degree $d<n$. More precisely,
\[
 A_n=C_n(f) = \Bigl [ a_{(r-s) \mod{n}} + a_{(r-s) \mod{n} - n} \Bigr ]_{r,s=0}^{n-1}.
\]
Due to $a_j=\overline{a_{-j}}$, $j=-d, \dots, d$, the matrix $C_n(f)$ is hermitian. 
Indeed, denote by $F_n \in \CC^{n \times n}$ the Fourier matrix of order $n$
\[
F_n = \frac{1}{\sqrt{n}} \Bigl [ e^{- \mathrm{i} \frac{2 \pi r s}{n}} \Bigr ]_{r,s = 0}^{n-1} \in \CC^{n \times n}.
\]
It is well known that any circulant matrix $C_n(f) \in \CC^{n \times n}$
satisfies
\begin{equation} \label{eq:circulant_diag}
C_n (f) = F_n \Delta_n (f) F_n^H, \quad \Delta_n (f) = \underset{r=0, \dots, n-1}{\text{diag}} f \left ( x_r^{(n)} \right ) \in \CC^{n \times n}, \quad x_r^{(n)} = \frac{2 \pi r}{n}.
\end{equation}
Hence, due to \eqref{eq:circulant_diag}, if $f \geq 0$, then $C_n(f)$ is symmetric and positive semi-definite. In particular, $C_n (f)$ is singular, if $f(x_r^{(n)})=0$ for some $x_r^{(n)}$, $r \in \{0,\ldots,n-1\}$.
In the latter case, the matrix $A_n$ can be defined as a sum of $C_n (f)$ and a rank one correction such that $A_n$ is positive definite. Such correction, due to Strang, has been considered in the convergence analysis in~\cite{simax}. However, it leads only to unnecessary complication of the notation, since the convergence results are not affected by such rank one correction. Moreover, in applications, $A_n$ is usually positive definite due to incorporated boundary conditions. Therefore, similarly to the analysis based on the LFA, the successive papers on the convergence analysis of multigrid methods for circulant matrices have neglected such a correction (see e.g. \cite{V-cycle_opt}).  We follow this standard approach and refer the interested reader to \cite{simax} for more details on rank one corrections.

In the case of circulant system matrices $A_n$, the grid transfer operators $P_{n_j}$ also
have a special structure. Let $\ell \in \NN, \, 1 \leq \ell \leq k-1$ and define
\begin{equation} \label{eq:vcycle_notation}
 P_{n_j} = C_{n_j} (p) \, K_{n_j,g}^T \in \CC^{n_j \times n_{j+1}}, \qquad n_j = g^{k-j},  
 \qquad j = 0, \dots, \ell-1,
\end{equation}
where $p$ is a certain trigonometric polynomial and $K_{n_j,g}\in \CC^{n_{j+1} \times n_j}$ is the \emph{downsampling matrix} of factor $g$
\begin{equation} \label{eq:down_sampl}
K_{n_j,g} = \quad \begin{bmatrix}
1 & 0_{g-1} \\
  & & 1 & 0_{g-1} \\
  & & & & \ddots & \\
  & & & & & 1 & 0_{g-1}
\end{bmatrix}.
\end{equation}
The operator $K_{n_j,g}$ allows to express $F_{n_{j+1}}$ in terms of $F_{n_j}$, see \cite{BIT2011}, i.e $F_{n_j}$ and $K_{n_j,g}$ satisfy the following packaging property
\begin{equation} \label{eq:pack_freq}
K_{n_j,g} \, F_{n_j} = \frac{1}{\sqrt{g}} \,
\Bigl [ \, \underbrace{\begin{matrix} F_{n_{j+1}} & | & \dots & | & F_{n_{j+1}} \end{matrix}}_{g \text{ times}} \, \Bigr ]
\in \CC^{n_{j+1} \times n_j}.
\end{equation}
This simple relation is the key step in defining multigrid methods for circulant matrices, since it allows us to obtain circulant matrices
$A_{n_{j+1}}$ at the lower levels. In fact, denote the set of \emph{g-corners} of $x \in [0, 2\pi )$ by
\begin{equation} \label{eq:corner_set}
\Omega_g (x) = \Set{ x + \frac{2 \pi j}{g} \pmod{2 \pi} \, : \, j=0, \dots, g-1}.
\end{equation}
It has been proved in \cite{BIT2011} that
\begin{equation} \label{def:Anj_circulant}
\begin{split}
&A_{n_{j+1}} = P_{n_j}^H A_{n_j} P_{n_j} = C_{n_{j+1}} (f_{j+1}), \\
f_{j+1} (x) &= \frac{1}{g} \sum_{y \in \Omega_g (\frac{x}{g})} f_{j} (y) \abs{p(y)}^2, \qquad x \in [0, 2\pi ),
\end{split}
\end{equation}
where $f_j$ are the trigonometric polynomials associated with the circulant matrices $A_{n_j} = C_{n_j} (f_j)$, $j=0, \dots, \ell$, and $f_0 = f$. 

\begin{remark}\label{rem:connection_to_subdivision}
Note that Step 2. in \eqref{alg:mgm} can be interpreted as the lowpass branch of a wavelet decomposition. At each level $n_j$, $j=0, \ldots, \ell-1$, the convolution with the lowpass filter is the multiplication 
by the matrix $C_{n_j}(p)^H$ and the downsampling by $g$ is done via multiplication
by the matrix $K_{n_j,g}$. If the smoother works well, then the residual is smooth and the highpass branches of the wavelet decomposition
contain no additional information and are omitted. The reconstruction is done as usual by upsampling via multiplication by $K_{n_j,g}^T$
and by convolution via multiplication by $C_{n_j}(p)$. It is well-known that upsampling and convolution amount to one step of subdivision
scheme with the corresponding subdivision matrix $P_{n_j}$. It is then natural to study conditions on the subdivision symbols $p$ that
will guarantee convergence and optimality of the corresponding multigrid methods.
\end{remark}

\section{Properties of multigrid methods for circulant matrices and subdivision}\label{sec:conv}

\vspace{0.3cm} In this section, we assume that the system matrix $A_n \in \CC^{n \times n}$ of the linear system \eqref{eq:sys} is circulant.
In subsection \ref{sec:main_result_TGM}, we exhibit a new big class of TGM grid transfer operators $P_n$ defined from
symbols of subdivision schemes with certain polynomial generation properties, see Theorem \ref{t:subd_tgm}.
In subsection \ref{sec:V_cycle_subdivision}, we derive sufficient conditions on the symbols of the multigrid grid transfer operators and recast them in subdivision terms.

\subsection{Two grid method} \label{sec:main_result_TGM}

\vspace{0.3cm} To be able to establish the link between the approximation property \eqref{eq:approx_property} and properties of subdivision schemes, we
 first relax the assumptions of \cite[Theorem 5.1]{BIT2011} for general $g \geq 2$ following 
 the analysis in \cite{BIT2014}.
 These conditions are easy to check for any given grid transfer operator $P_n$.
 Our simplification in Theorem \ref{t:tgm_new} replaces $(ii)$ in \cite[Theorem 5.1]{BIT2011} by an even simpler condition, see
  $(ii)$ in Theorem~\ref{t:tgm_new}.

We denote the set of \emph{g-mirror points} of $x \in [0, 2\pi )$ by
\[
M_g (x): = \Omega_g (x) \setminus \{ x \} = \Set{x + \frac{2 \pi j}{g} \pmod{2 \pi} \, : \, j=1, \dots, g-1}.
\]

\begin{theorem} \label{t:tgm_new}
Let $f$ and $p$ be real trigonometric polynomials such that $f(x_0) = 0$ and $f(x) > 0$, $x \in [0,2\pi) \setminus \{ x_0 \}$.
If $p$ satisfies
\[
\begin{array}{ll} \bigskip
 (i) & {\displaystyle \lim_{x \to x_0} \frac{\abs{p(y)}^2}{f(x)}} < +\infty \qquad \forall \, y \in M_g (x_0), \\
(ii) & \abs{p(x_0)}^2 > 0,
\end{array}
\]
then $P_{n} = C_n (p) \, K_{n,g}^T$ satisfies the \emph{approximation property} \eqref{eq:approx_property}.
\end{theorem}

\begin{proof}
The proof consists of three steps. The first and second steps are borrowed from \cite{BIT2011} and \cite{Serra_Tablino}, thus we only state them
shortly. We present in  detail the proof of the main step, \emph{3. step}.

\noindent \emph{1. step:} Let $ {\displaystyle a_0 = \frac{1}{2 \pi} \int_{0}^{2\pi} f(x) dx }$. By Theorem 5.1  in \cite{BIT2011}, \eqref{eq:approx_property} is equivalent to
\begin{equation} \label{eq:new_thesis}
\exists \, \gamma >0 \quad \hbox{independent of $n$ such that} \quad I_n - P_n ( P_n^H P_n )^{-1} P_n^H \preceq \frac{\gamma}{a_0} C_n(f),
\end{equation}
i.e. the matrix $\frac{\gamma}{a_0} C_n(f) - I_n + P_n ( P_n^H P_n )^{-1} P_n^H$ is positive semi-definite.

\noindent \emph{2. step:} For $x \in [0, 2\pi )$, let
$$
 y_j = y_j(x) := x + \frac{2 \pi j}{g} \pmod{2 \pi}, \quad j=0, \dots,g-1,
$$
be the elements of the $g$-corner set $\Omega_g(x)$. Define the row vectors $p[x], \, f[x] \in \CC^{1 \times g}$ by
\[
 p[x]:=\begin{bmatrix} p(y_0) & \dots & p(y_{g-1})\end{bmatrix} \quad \hbox{and} \quad f[x]:= \begin{bmatrix} f(y_0) & \dots & f(y_{g-1}) \end{bmatrix}.
\]
By Theorem 5.1  in \cite{BIT2011}, \eqref{eq:new_thesis} is equivalent to
\begin{equation} \label{eq:new_thesis_2}
 \exists \, \gamma >0 \quad \hbox{independent of $n$ such that} \quad I_g - \frac{ p[x]^H \, \cdot p[x] }{\norma*{p[x]}_2^2} \, \preceq \, \frac{\gamma}{a_0} \, \text{diag } \left (f[x] \right ),
\end{equation}
$\forall \, x \in [0, 2\pi )$.

\noindent \emph{3. step:} We follow the approach of Bolten et al. in \cite{BIT2014}. To prove the claim, we show that assumptions $(i)$ and $(ii)$ imply \eqref{eq:new_thesis_2}. 
To do so, we need to show that the $g \times g$ matrix
\[
R[x] := \left ( \, \text{diag } (f[x]) \right )^{-\frac{1}{2}} \, \left ( I_g - \frac{ p[x]^H \, \cdot p[x] }{\norma*{p[x]}_2^2} \right ) \left ( \, \text{diag } (f[x]) \right )^{-\frac{1}{2}}
\]
is well-defined, i.e. we can bound the modulus of its entries $R[x]_{r,s}$ by
\begin{equation} \label{eq:Rx_well_defined}
\abs*{R[x]_{r,s}} \leq \frac{\gamma}{a_0} < \infty \qquad \forall x \in [0, 2\pi ), \qquad r,s=0, \dots, g-1.
\end{equation}
Note that, for $r,s=0, \dots, g-1$, the entries $R[x]_{r,s}$ of $R[x]$ are given by
\begin{equation} \label{eq:Rx_entries}
\begin{aligned}
R[x]_{r,s} &= - \frac{p(y_r) \overline{p(y_s)}}{\sqrt{f(y_r) f(y_s)} {\displaystyle \sum_{y \in \Omega_g(x)} \abs{p(y)}^2 }}, &\qquad r \neq s, \\
\bigskip
R[x]_{s,s} &= \frac{ {\displaystyle \sum_{y \in M_g(y_s)} \abs{p(y)}^2}}{ f(y_s) {\displaystyle \sum_{y \in \Omega_g(x)} \abs{p(y)}^2 } }, &\qquad r = s.
\end{aligned}
\end{equation}
In the following, we consider two cases, $x \in \Omega_g(x_0)$ and $x \not\in \Omega_g(x_0)$.
If $x \in \Omega_g(x_0)$, then by the definition in \eqref{eq:corner_set}, $\Omega_g(x_0) = \Omega_g(x)$.
Moreover, if $x \in \Omega_g(x_0)$, then $\exists \, s \in \set{0, \dots, g-1}$ such that $y_s = x_0$ and $f(y_s) = 0$. 
If $r \neq s$, then by $(i)$, the order of the zero of $\sqrt{f}$ at $y_s$ matches the order of the zero of $p$ at $y_r$. 
If $r=s$, then again by $(i)$ with
$$
  \sum_{y \in M_g(y_s)} \abs{p(y)}^2= \sum_{y \in M_g(x_0)} \abs{p(y)}^2,
$$
the order of the zero of $f$ at $y_s$ matches the order of the zero of ${ \displaystyle \sum_{y \in M_g(y_s)} } \abs{p(y)}^2$.
It is left to show that ${\displaystyle \sum_{y \in \Omega_g(x)} \abs{p(y)}^2 } > 0$ for any $r,s=0, \ldots, g-1$, then all entries of $R[x]$, $x \in \Omega_g(x_0)$, are well-defined.
The identity $\Omega_g(x_0) = \Omega_g(x)$ and $(i)$ imply that $p(y)=0$ for all $y \in M_g(x_0)$. Thus,
by $(ii)$, we get
$$
 {\displaystyle \sum_{y \in \Omega_g(x)}} \abs{p(y)}^2 = {\displaystyle \sum_{y \in \Omega_g(x_0)}} \abs{p(y)}^2 =\abs{p(x_0)}^2 >0.
$$

\noindent We assume next that $x \notin \Omega_g(x_0)$. First, we notice that if $x \notin \Omega_g(x_0)$, then $x_0 \notin \Omega_g(x)$ and $f(y_s) \neq 0$, $s=0, \dots, g-1$, since $f$ has a unique zero at $x_0$ by hypothesis. Thus, we only need to study the properties of ${\displaystyle \sum_{y \in \Omega_g(x)} \abs{p(y)}^2 }$.

If ${\displaystyle \sum_{y \in \Omega_g(x)} \abs{p(y)}^2 } > 0$, then $\abs{R[x]_{r,s}} < \infty$ for any $r,s=0,\dots,g-1$.

If ${\displaystyle \sum_{y \in \Omega_g(x)} \abs{p(y)}^2 } = 0$, then we need to study the behaviour of its zeros. To do that
we first define, for a trigonometric polynomial $h$, the function $\theta_h \colon \RR \to \NN$ such
that 
\begin{equation}\label{eq:theta}
\theta_h (\bar{x})= m    \qquad
\Longleftrightarrow \qquad
D^{\mu} h (\bar{x}) = 0, \qquad \mu = 0, \ldots, m-1, \qquad D^m h (\bar{x}) \neq 0,
\end{equation}
i.e. $m \in \NN$ is the order of the zero of $h$ at $\bar{x}$.
We rewrite the entries $R[x]_{r,s}$ $r,s=0, \dots,g-1$, of $R[x]$ in \eqref{eq:Rx_entries} and get
\[
\begin{aligned}
R[x]_{r,s} &= - \frac{h_{r,s}(x)}{\sqrt{f(y_r) f(y_s)} \, h(x)}, &\qquad r \neq s, \\
\bigskip
R[x]_{s,s} &= \frac{ h_s (x)}{ f(y_s) \, h(x) }, &\qquad r = s,
\end{aligned}
\]
where
\[
\begin{split}
h(x) &:= \sum_{y \in \Omega_g(x)} \abs{p(y)}^2 = \sum_{j=0}^{g-1} \abs*{p \left (x + 2 \pi j /g  \right )}^2, \\
h_s(x) &:= \sum_{y \in M_g(y_s)} \abs{p(y)}^2 = \sum_{j=0, j \neq s}^{g-1} \abs*{p \left (x + 2 \pi j / g \right) }^2  , \\
h_{r,s} (x) &:= p(y_r) \overline{p(y_s)} = p \left (x + 2 \pi r / g \right) \overline{p \left (x + 2 \pi s / g \right)}.
\end{split}
\]
To prove the boundedness of $R[x]_{r,s}$ $r,s=0, \dots,g-1$, we show that
\[
\theta_{h_s}(x) \geq \theta_h(x), \qquad \text{and} \qquad \theta_{h_{r,s}}(x) \geq \theta_h(x).
\]
Recall that we consider the case when $h(x) = 0$, then $p(y) = 0$ for all $y \in \Omega_g (x)$. Thus, for $\Theta := {\displaystyle \min_{y \in \Omega_g(x)}} \theta_p (y)$, we have $\theta_h(x) = 2 \Theta$.  
Due to $M_g (y_s) \subset \Omega_g (x)$ we get $\theta_{h_s} (x) \geq 2 \Theta$. Similarly, $\theta_{h_{r,s}}(x) \geq 2 \Theta$. And, thus, the claim follows.
\end{proof}

\begin{remark}
If $f$ has an additional zero at some point $x_1 \notin M_g(x_0)$, then we choose a trigonometric polynomial $p$ which satisfies  $(i)$ of Theorem \ref{t:tgm_new} for $y \in M_g(x_0) \cup M_g(x_1)$  and $(ii)$ of Theorem \ref{t:tgm_new} for $x_0$ and $x_1$.
Then, the corresponding $P_{n} = C_n (p) \, K_{n,g}^T$ also satisfies the \emph{approximation property} \eqref{eq:approx_property}.  The proof of the latter is a straightforward generalization of the proof of Theorem \ref{t:tgm_new} and is omitted. If $f$ has an additional zero at some point $x_1 \in M_g(x_0)$, then we choose a different down-sampling factor $\tilde{g} \geq 2$,  $\tilde{g} \neq g$, so that $x_1 \notin M_{\tilde{g}}(x_0)$.
\end{remark}

Under the assumption that the trigonometric polynomial $f$ has a zero at $x_0=0$, Theorem \ref{t:tgm_new} has an equivalent subdivision
formulation, see Theorem \ref{t:subd_tgm}. To state Theorem \ref{t:subd_tgm}, we use Laurent polynomial formalism and
talk about the subdivision symbol $p$ in \eqref{eq:symbol_trigo_poly}.

\begin{theorem} \label{t:subd_tgm}
Let $f$ be a real trigonometric polynomial such that $f(x) > 0$, $x \in (0,2\pi)$, and $D^\mu f(0)=0$, 
$\mu=0,\ldots,m-1$, $D^{m}f(0) \neq 0$. Assume that the subdivision scheme $S_{\mathbf{p}}$ of arity $g$ and with symbol
$$
 p(z) = \mathrm{p}_0 + \sum_{\alpha \in \NN} \mathrm{p}_{\alpha} \left (z^{-\alpha} + z^{\alpha} \right), \qquad z \in \CC \setminus \set{0},
$$
is convergent. If $S_{\mathbf{p}}$ generates polynomials up to degree $\lceil \frac{m}{2} \rceil - 1$, then
the corresponding grid transfer operator $P_n$ satisfies the approximation property \eqref{eq:approx_property}.
\end{theorem}

\begin{proof}
By Theorem \ref{t:poly_generation} and due to convergence of $S_{\mathbf{p}}$, the symbol $p$ satisfies for $ q \geq \lceil \frac{m}{2} \rceil - 1$
\begin{equation} \label{aux}
\begin{array}{ll}
 (i) & D^\mu p (\varepsilon) = 0 \quad  \mu = 0, \dots, q, \quad \forall \, \varepsilon \in
 \{e^{-i\frac{2\pi j}{g}}\ : \ j=1, \ldots, g-1\}, \\
(ii) & p(1) = g.
\end{array}
\end{equation}
To prove the claim, we show that $(i)$ and $(ii)$ in \eqref{aux} imply conditions $(i)$ and $(ii)$ of Theorem \ref{t:tgm_new}.
For $z=e^{-ix}$, $x \in \RR$, the polynomial $p$ is a real trigonometric polynomial. Thus, we write
$p(y) := p(e^{-iy})$, $y \in [0,2\pi)$. From
\[
 E_g \setminus \set{1} = \{ e^{-iy} \, : \, y \in M_g(0) \}
\]
conditions $(i)$ and $(ii)$ in \eqref{aux} become
\[
\begin{array}{ll}
 (i) & D^\mu p (y) = 0 \qquad  \mu = 0, \dots, q, \qquad q \geq \lceil \frac{m}{2} \rceil - 1, \qquad \forall \, y \in M_g(0), \\
 (ii) & p(0) = g,
\end{array}
\]
which imply assumptions $(i)$ and $(ii)$ of Theorem \ref{t:tgm_new}.
\end{proof}

\subsection{Properties of V-cycle and subdivision} \label{sec:V_cycle_subdivision}

\vspace{0.3cm} 
For the V-cycle, according to the convergence and optimality results in \cite{simax}, the assumptions of Theorem \ref{t:tgm_new} should be strengthen 
to guarantee that the corresponding coarse grid correction operators satisfy the approximation property \eqref{eq:approx_property_vcycle}.
The appropriate modifications of the assumptions of Theorem \ref{t:tgm_new} were given in \cite{V-cycle_opt} for $g=2$.
The following Theorem \ref{t:vcycle} is the generalization of Theorem \ref{t:tgm_new} to the case
$g > 2$.

\begin{theorem} \label{t:vcycle}
Let $f_0, \, p_j$, $j=0, \ldots, \ell-1$, be real trigonometric polynomials such that $f_0(x_0) = 0$ and $f_0(x) > 0$, $x \in [0,2\pi) \setminus \{ x_0 \}$. Let $f_j$, $j=1, \ldots, \ell-1$, be real trigonometric polynomials defined 
as in \eqref{def:Anj_circulant} and such that $f_j(x_j) = 0$, $f_j(x) > 0$, $x \in [0,2\pi) \setminus \{ x_j \}$.
If, for $j=0, \ldots, \ell-1$, $p_j$ satisfy
\[
\begin{array}{llcll} \bigskip
 (i) & {\displaystyle \lim_{x \to x_j} \frac{\abs{p_j(y)}}{f_j(x)}} &<& +\infty & \qquad \forall \, y \in M_g (x_j), \\
 (ii) & {\displaystyle \sum_{y \in \Omega_g(x)}} \abs{p_j(y)}^2 &>& 0 & \qquad \forall \, x \in [0, 2\pi ),
\end{array}
\]
then $A_{n_j}$ in \eqref{def:Anj_circulant} and $CGC_{n_j}$ in \eqref{def:CGC_nj} satisfy the approximation property \eqref{eq:approx_property_vcycle}.
\end{theorem}

Before proving Theorem \ref{t:vcycle}, we would like to comment on its hypothesis. Let $j \in \{0, \ldots, \ell-2\}$. If $f_j(x_j) = 0$, $f_j(x) > 0$ for $x \in [0,2\pi) \setminus \{ x_j \}$ and $p_j$ satisfies $(i)$ and $(ii)$ of
Theorem \ref{t:vcycle}, then \cite[Proposition 4.1]{BIT2011} guarantees that $f_{j+1}(x) = 0$ if and only if 
$x = x_{j+1} := g x_j \pmod{2 \pi}$.  Moreover, the order of the zero of $f_{j+1}$ at $x_{j+1}$ coincides with the order of 
the zero of $f_j$ at $x_j$ and $f_{j+1}(x) > 0$ for $x \in [0,2\pi) \setminus \{ x_{j+1} \}$.

\begin{proof}
The proof consists of two steps: the first one is borrowed from \cite{simax}, the second one is similar to \emph{3. step} of the proof of Theorem \ref{t:tgm_new}. Let $j \in \{0, \ldots, \ell-1\}$.

\noindent \emph{1. step:} By  \cite[Proposition 16]{simax}, $A_{n_j}$ in \eqref{def:Anj_circulant} and $CGC_{n_j}$ in \eqref{def:CGC_nj} satisfy the approximation property \eqref{eq:approx_property_vcycle} if and only if
\begin{equation} \label{eq:new_thesis_Vcycle}
\exists \, \gamma_j >0 \quad \hbox{independent of $n_j$ such that} \quad I_{n_j} - \hat{P}_{n_j} ( \hat{P}_{n_j}^H \hat{P}_{n_j} )^{-1} \hat{P}_{n_j}^H \preceq \gamma_j C_{n_j}(f),
\end{equation}
where $\hat{P}_{n_j} := C_{n_j} (\hat{p}_j) \, K_{n_j,g}^T \in \CC^{n_j \times n_{j+1}}$, and $\hat{p}_j (x) := p_j (x)  \sqrt{ f_j (x) }$, $x \in [0, 2\pi)$.

\noindent \emph{2. step:}  To prove the claim, we show that $(i)$ and $(ii)$ imply \eqref{eq:new_thesis_Vcycle}. As shown in \emph{3. step} of the proof of Theorem \ref{t:tgm_new}, \eqref{eq:new_thesis_Vcycle} holds true 
if and only if the entries of the matrix $R[x]$ in \eqref{eq:Rx_entries} are bounded in modulus, where, for 
$y_r,y_s \in \Omega_g(x)$, $x \in [0, 2\pi)$, $r,s = 0, \ldots, g-1$,
\begin{equation} \label{eq:Rx_entries_new}
\begin{aligned}
R[x]_{r,s} &= - \frac{\hat{p_j}(y_r) \overline{\hat{p_j}(y_s)}}{\sqrt{f_j(y_r) f_j(y_s)} {\displaystyle \sum_{y \in \Omega_g(x)} \abs{\hat{p_j}(y)}^2 }}, &\qquad r \neq s, \\
\bigskip
R[x]_{s,s} &= \frac{ {\displaystyle \sum_{y \in M_g(y_s)} \abs{\hat{p_j}(y)}^2}}{ f_j(y_s) {\displaystyle \sum_{y \in \Omega_g(x)} \abs{\hat{p_j}(y)}^2 } }, &\qquad r = s.
\end{aligned}
\end{equation}
Substituting the definition of $\hat{p}_j$ into \eqref{eq:Rx_entries_new}, we get
\begin{equation} \label{eq:Rx_entries_new2}
\begin{aligned}
R[x]_{r,s} &= - \frac{p_j(y_r) \overline{p_j(y_s)}}{{\displaystyle \sum_{y \in \Omega_g(x)} \abs{p_j(y)}^2 f_j(y) }}, &\qquad r \neq s, \\
\bigskip
R[x]_{s,s} &= \frac{ {\displaystyle \sum_{y \in M_g(y_s)} \abs{p_j(y)}^2 f_j(y)}}{ f_j(y_s) {\displaystyle \sum_{y \in \Omega_g(x)} \abs{p_j(y)}^2 f_j(y)} }, &\qquad r = s.
\end{aligned}
\end{equation}
We split the analysis of quantities in \eqref{eq:Rx_entries_new2} into two cases: $x \in \Omega_g (x_j)$ and $x \notin \Omega_g (x_j)$.

If $x \in \Omega_g (x_j)$, then by the definition in \eqref{eq:corner_set}, $\Omega_g(x_j) = \Omega_g (x)$. Thus, the hypothesis 
$f_j(x_j) = 0$ and $(i)$ imply that
\begin{equation} \label{aux1000}
\begin{split}
\sum_{y \in \Omega_g(x)} \abs{p_j(y)}^2 f_j(y) &= \sum_{y \in \Omega_g(x_j)} \abs{p_j(y)}^2 f_j(y) \\
&= \abs{p_j(x_j)}^2 f_j(x_j) + \sum_{y \in M_g(x_j)} \abs{p_j(y)}^2 f_j(y) = 0.
\end{split}
\end{equation}
We define 
\[
\begin{aligned}
h(x) &:= \sum_{y \in \Omega_g(x)} \abs{p_j(y)}^2 f_j(y), \\
h_s(x) &:= \sum_{y \in M_g(y_s)} \abs{p_j(y)}^2 f_j(y), \\
h_{f_j,s} (x) &:= f_j(y_s) \sum_{y \in \Omega_g(x)} \abs{p_j(y)}^2 f_j(y) = f_j(x + 2 \pi s / g) \sum_{y \in \Omega_g(x)} \abs{p_j(y)}^2 f_j(y), \\ 
h_{r,s} (x) &:= p_j(y_r) \overline{p_j(y_s)}.
\end{aligned}
\]
Then, we can rewrite $R[x]_{r,s}$, $r, s = 0, \ldots, g-1$, as
\[
\begin{aligned}
R[x]_{r,s} = - \frac{ h_{r,s} (x) }{ h(x) }, \quad r \neq s,  \quad \hbox{and} \quad R[x]_{s,s} = \frac{ h_s(x) }{ h_{f_j,s} (x) }.
\end{aligned}
\]
To prove the boundedness of $R[x]_{r,s}$ $r,s=0, \dots,g-1$, we show, for $\theta$ as in \eqref{eq:theta}, that
\[
 \theta_{h_{r,s}}(x) \geq \theta_h(x) \qquad \text{and} \qquad \theta_{h_s}(x) \geq \theta_{h_{f_j,s}} (x).
\]
Note  first that $(i)$ and \eqref{aux1000} guarantee that the order of the zero of $h$ at $x$ is the same as the order of the zero 
of $f_j$ at $x_j$. Namely, for $\Theta:=\theta_{f_j} (x_j)$, we have $\theta_h (x)= \Theta$. Due to $(i)$, 
$\theta_{h_{r,s}}(x) \geq \Theta$. Thus, $\theta_{h_{r,s}}(x) \geq \theta_h(x)$.
Since $x \in \Omega_g (x_j)$, there exists $\bar{s} \in \set{0, \ldots, g-1}$ such that $y_{\bar{s}} = x_j$. If $s = \bar{s}$, then, 
by $(i)$ and \eqref{aux1000}, $\theta_{h_s}(x) \geq \theta_{h_{f_j,s}} (x) = 2 \Theta$. Otherwise, 
$\theta_{h_s}(x) = \theta_{h_{f_j,s}} (x) = \Theta$. 

We assume next that $x \notin \Omega_g(x_j)$. First, we notice that, if $x \notin \Omega_g(x_j)$, then $x_j \notin \Omega_g(x)$.
Since $f_j$ has a unique zero at $x_j$ by hypothesis, we have $f_j(y_s) \neq 0$, $s=0, \dots, g-1$.  Thus, we only need to study the properties of ${\displaystyle \sum_{y \in \Omega_g(x)} \abs{p_j(y)}^2 f_j(y) }$.  Since $f_j$ has a unique zero at $x_j$ by hypothesis, by $(ii)$,
we obtain
\[
\sum_{y \in \Omega_g(x)} \abs{p_j(y)}^2 f_j(y) > 0.
\]
And, thus, the claim follows.
\end{proof}

If $f_0(0) = 0$ and $f_0(x) > 0$, $x \in (0, 2\pi)$, then  \cite[Proposition 4.1]{BIT2011} guarantees that every $f_j$, 
$j=1, \ldots, \ell-1$, vanishes only at 0 with the same order  as the one of the zero of $f_0$.
Thus, we use $p_j = p$, $j=0, \ldots, \ell-1$. If $p$ satisfies
\[
\lim_{x \to 0} \frac{\abs{p(y)}}{f_0(x)} < +\infty \qquad \forall \, y \in M_g (0), 
\]
then condition $(i)$ of Theorem \ref{t:vcycle} is satisfied. We, thus, focus on the case $x_0=0$, since it is of practical interest, see e.g.
Examples \ref{sec:examples_biharmonic} and \ref{sec:examples_laplacian}.

Recall, from \eqref{def:CGC_nj}, that one of the main ingredients in the definition of $CGC_{n_j}$ are the grid transfer
operators $P_{n_j} = C_{n_j}(p) K_{n_j,g}^T$. We view again $p$ as the symbol of a convergent subdivision scheme $S_{\mathbf{p}}$.
Our goal is to identify subdivision schemes $S_{\mathbf{p}}$ whose symbols $p$ satisfy assumptions of Theorem \ref{t:vcycle}
for $x_0=0$.

\begin{theorem} \label{t:subd_vcycle}
Let $f$ be a real trigonometric polynomial such that $f(x) > 0$, $x \in (0,2\pi)$, and $D^\mu f(0) = 0$, $\mu=0,\ldots,m-1$, $D^m f(0) \neq 0$.
Assume that the subdivision scheme $S_{\mathbf{p}}$ of arity $g$ and with symbol
$$
 p(z) = \mathrm{p}_0 + \sum_{\alpha \in \NN} \mathrm{p}_{\alpha} \left (z^{-\alpha} + z^{\alpha} \right), \qquad z \in \CC \setminus \set{0},
$$
is convergent. If
\begin{enumerate}
 \item[(i)] $S_{\mathbf{p}}$ generates polynomials up to degree $m-1$,
 \item[(ii)] the basic limit function $\phi$ of $S_{\mathbf{p}}$ is $\ell^{\infty}$-stable,
\end{enumerate}
then the approximation property \eqref{eq:approx_property_vcycle} is satisfied.
\end{theorem}

\begin{proof}
To prove the claim we show that condition $(i)$ is equivalent to $(i)$ of Theorem \ref{t:vcycle} and that property $(ii)$ implies $(ii)$ of Theorem \ref{t:vcycle}. The equivalence of $(i)$ follows by the same argument as in the proof of Theorem \ref{t:subd_tgm}. Next we show that, if the basic function $\phi$ has $\ell^{\infty}$-stable integer translates, then condition $(ii)$ of Theorem \ref{t:vcycle} is satisfied.
Define the Fourier transform of a continuous, compactly supported $\phi$ by
\[
 \hat{\phi}(x) = \int_{\RR} \phi(t) e^{-\mathrm{i} t \, x} dt, \qquad x \in \RR.
\]
Define also
\[
\Pi_{\phi} (x) = \sum_{\alpha \in \ZZ} \abs{\hat{\phi} (x + 2 \pi \alpha) }^2, \qquad x \in \RR.
\]
Note that, due to the Poisson summation formula, we have
\[
\Pi_{\phi} (x) = \sum_{\alpha \in \ZZ} d_\alpha e^{- \mathrm{i} \alpha x}, \qquad d_\alpha = \int_{\RR} \phi(t) \phi(t-\alpha) dt, \qquad x \in \RR.
\]
The compact support of $\phi$ implies that $\Pi_{\phi}$ is a trigonometric polynomial. Next, we take the Fourier transforms of both sides of the refinement equation \eqref{eq:ref_eq}
and obtain
\[
\hat{\phi} (x) = \frac{1}{g} p \left ( e^{- \mathrm{i} \frac{x}{g}} \right ) \, \hat{\phi} \left ( \frac{x}{g} \right ), \qquad x \in \RR.
\]
Then, following the steps in \cite{stability_arity_g}, we write $\alpha = j+g \beta$, $j=0, \ldots, g-1$, $\beta \in \ZZ$, and get
\[
\begin{split}
\Pi_{\phi} (x) &= \sum_{\alpha \in \ZZ} \abs{\hat{\phi} (x + 2 \pi \alpha) }^2 \\
		    &= \sum_{\alpha \in \ZZ} \frac{1}{g^2} \abs*{p \left ( e^{- \mathrm{i} \frac{x + 2 \pi \alpha}{g}} \right )}^2 \, \abs*{\hat{\phi} \left ( \frac{x + 2 \pi \alpha}{g} \right )}^2 \\
		    &= \sum_{j=0}^{g-1} \frac{1}{g^2} \abs*{p \left ( e^{- \mathrm{i} \frac{x + 2 \pi j}{g}} \right )}^2 \, \sum_{\beta \in \ZZ} \abs*{\hat{\phi} \left ( \frac{x + 2 \pi (j + \beta g) }{g} \right )}^2 \\
		    &= \sum_{j=0}^{g-1} \frac{1}{g^2} \abs*{p \left ( e^{- \mathrm{i} \frac{x + 2 \pi j}{g}} \right )}^2 \, \sum_{\beta \in \ZZ} \abs*{\hat{\phi} \left ( \frac{x + 2 \pi j}{g} + 2 \pi \beta \right )}^2 \\
		    &= \sum_{j=0}^{g-1} \frac{1}{g^2} \abs*{p \left ( e^{- \mathrm{i} \frac{x + 2 \pi j}{g}} \right )}^2 \, \Pi_{\phi} \left ( \frac{x + 2 \pi j}{g} \right ).
\end{split}
\]
It was proved in \cite{Jia_Micc_stability} that a continuous, compactly supported $\phi$ has $\ell^{\infty}$-stable integer translates if and only if
\begin{equation} \label{eq:stability_sup}
 \sup_{\alpha \in \ZZ} \, \abs{\hat{\phi} (x + 2 \pi \alpha) } > 0, \qquad \forall \, x \in \RR.
\end{equation}
This is equivalent to $\Pi_{\phi} (x) > 0, \, \forall \, x \in \RR$.
Thus, we have
\[
\sum_{j=0}^{g-1} \abs*{p \left ( e^{- \mathrm{i} \frac{x + 2 \pi j}{g}} \right )}^2 > 0, \qquad \forall \, x \in \RR.
\]
Since, for $z=e^{-ix}$, $x \in \RR$, the polynomial $p$ is a trigonometric polynomial, we write $p(x) := p(e^{-ix})$, $x \in [0,2\pi)$. Thus,
the claim follows, by the definition of the $g$-corner set $\Omega_g$ in \eqref{eq:corner_set},
\[
\sum_{j=0}^{g-1} \abs*{p \left ( \frac{x + 2 \pi j}{g} \right )} = \sum_{j=0}^{g-1} \abs*{p \left ( \frac{x}{g} + \frac{2 \pi j}{g} \right )}
 =\sum_{y \in \Omega_g(\frac{x}{g})} \abs{p(y)}^2 > 0, \qquad \forall \, x \in \RR.
\]
Therefore, $(ii)$ of Theorem \ref{t:vcycle} is also satisfied.
\end{proof}

If $\phi$ is not given explicitly or $(ii)$ of Theorem \ref{t:subd_vcycle} is difficult to check, one can use an alternative criterion which guarantees the validity of condition \emph{(ii)}
of Theorem \ref{t:vcycle}.

\begin{proposition} \label{p:cohen_simple_vcycle}
Let $p$ be a trigonometric polynomial and $g \in \NN$, $g \ge 2$. If
\[
\abs*{p \left( e^{- \mathrm{i} x} \right)} > 0, \qquad \forall \, x \in \left[-\frac{\pi}{g}, \frac{\pi}{g}\right],
\]
then
\begin{equation} \label{aux1}
\sum_{j=0}^{g-1} \, \abs*{p \left ( e^{- \mathrm{i} \left ( x + \frac{2 \pi j}{g} \right )} \right ) }^2 > 0, \qquad \forall \, x \in [0, 2\pi).
\end{equation}
\end{proposition}

\begin{proof}
To simplify the arguments, we first rewrite \eqref{aux1} in an equivalent way.  We use the substitution $ j' = j+1$ and get
\[
\begin{split}
&\sum_{j=0}^{g-1} \, \abs*{p \left ( e^{- \mathrm{i} \left ( x + \frac{2 \pi j}{g} \right )} \right ) }^2> 0, \qquad \forall \, x \in [0,2\pi),
\quad
\iff \\
&\sum_{j'=1}^{g} \, \abs*{p \left ( e^{- \mathrm{i} \left ( x + \frac{2 \pi j'}{g} \right )} \right ) }^2 > 0, \qquad \forall \, x \in \left [- \frac{2\pi}{g}, \frac{(2g-2) \pi}{g} \right ).
\end{split}
\]
Straightforwardly, the latter inequality is equivalent to
\[
\sum_{j=1}^{g} \, \abs*{p \left ( e^{- \mathrm{i} \left ( x + \frac{2 \pi j}{g} \right )} \right ) }^2 > 0, \qquad \forall \, x \in \left [- \frac{\pi}{g}, \frac{(2g-1) \pi}{g} \right ).
\]
Let $x \in \left [- \frac{\pi}{g}, \frac{(2g-1) \pi}{g} \right )$. There exists $ k \in \{ 0, \ldots, g-1\}$ such that $x \in \left [ \frac{(2k-1)\pi}{g}, \frac{(2k+1) \pi}{g} \right )$.
Define  $\ell= g - k$. Then $ \ell \in \set{1, \dots, g}$ and
$x + \frac{2 \pi \ell}{g} \in \left [ -\frac{\pi}{g}, \frac{\pi}{g} \right )$, due to
\begin{gather*}
\frac{(2k-1)\pi}{g} \leq \, x < \frac{(2k+1)\pi}{g} \\
\frac{(2g-1)\pi}{g} =\frac{(2k-1)\pi}{g} + \frac{2 \pi \ell}{g} \leq x + \frac{2 \pi \ell}{g} <\frac{(2k+1)\pi}{g} + \frac{2 \pi \ell}{g}=\frac{(2g+1)\pi}{g} \\
- \frac{\pi}{g} \pmod{2 \pi} \leq x + \frac{2 \pi \ell}{g} <  \frac{\pi}{g} \pmod{2 \pi}.
\end{gather*}
By hypothesis, we get
$ \displaystyle
\abs*{p \left ( e^{- \mathrm{i} \left ( x + \frac{2 \pi \ell}{g} \right ) } \right)} > 0
$, which yields the claim
\[
\sum_{j=1}^{g} \, \abs*{p \left ( e^{- \mathrm{i} \left ( x + \frac{2 \pi j}{g} \right )} \right ) }^2 > 0, \qquad \forall \, x \in [0, 2\pi).
\]
\end{proof}

Hypothesis of Proposition \ref{p:cohen_simple_vcycle} is a simplified version of the so-called \emph{Cohen's condition}. This condition was first introduced by Cohen in \cite{Cohen_thesis} and then it was analyzed in depth regarding wavelets
and orthonormality by Daubechies in \cite{Daub_ten_lectures}.

\begin{definition} \label{d:cohen_condition}
We say that a trigonometric polynomial $p$ satisfies Cohen's condition if there exists a compact set $K \subset \RR$ satisfying
\begin{enumerate}
\item[(i)] $0 \in K$,
\item[(ii)] $\abs{K} = 2 \pi$,
\item[(iii)] for all $x \in \RR$, there exists $\ell \in \ZZ$ such that $x + 2 \ell \pi \in K$,
\end{enumerate}
and such that there exists $k_0 > 0$ for which
\[
\abs*{p \left( e^{- \mathrm{i} x } \right)} > 0, \qquad \forall \, x \in  \bigcup_{j=1}^{k_0} g^{-j} K.
\]
\end{definition}

\begin{remark}
If a compact set $K \subset \RR$ satisfies conditions $(ii)$ and $(iii)$ in Definition \ref{d:cohen_condition}, we say that $K$ is congruent to $[-\pi,\pi]$ modulo $2 \pi$.
In Proposition \ref{p:cohen_simple_vcycle}, we require that the trigonometric polynomial $p$ satisfies Cohen's condition with the special choices $K = [-\pi,\pi]$, $k_0 = 1$.
\end{remark}

Finally, using the result of Proposition \ref{p:cohen_simple_vcycle}, we get the following result.

\begin{theorem} \label{t:cohen_simple_vcycle}
Let $f$ be a real trigonometric polynomial such that $f(x) > 0$, $x \in (0,2\pi)$, and $D^\mu f(0) = 0$, $\mu=0,\ldots,m-1$, $D^m f(0) \neq 0$.
If the symbol
$$
 p(z) = \mathrm{p}_0 + \sum_{\alpha \in \NN} \mathrm{p}_{\alpha} \left (z^{-\alpha} + z^{\alpha} \right), \qquad z \in \CC \setminus \set{0},
$$
satisfies
\begin{enumerate}
 \item[(i)] zero conditions of order $m$,
 \item[(ii)] $\abs*{p \left( e^{- \mathrm{i} x} \right)} > 0, \quad \forall \, x \in \left[-\frac{\pi}{g}, \frac{\pi}{g}\right],$
\end{enumerate}
then the approximation property \eqref{eq:approx_property_vcycle} is satisfied.
\end{theorem}

\begin{proof}
We have already shown in the proof of Theorem \ref{t:subd_vcycle} that assumption $(i)$ is equivalent to condition $(i)$ of Theorem \ref{t:vcycle}.
By Proposition \ref{p:cohen_simple_vcycle}, condition $(ii)$  implies \eqref{aux1}. Note that \eqref{aux1} is equivalent to \emph{(ii)} in Theorem \ref{t:vcycle}.
\end{proof}

\section{Grid transfer operators from primal pseudo-splines}\label{sec:splines}

In this section, we define grid transfer operators from well-known subdivision symbols of pseudo-splines introduced in \cite{Dau_Han_Ron_Shen_pseudo}. Recall that we only consider odd symmetric symbols, i.e. we restrict our attention to primal pseudo-splines. This is due to the use of vertex centered discretization in section \ref{sec:num_examples}. 
In section \ref{sec:examples_binary_ps}, we define and analyze grid transfer operators derived from binary pseudo-splines.
Then in section \ref{sec:examples_ternary_ps},  we use symbols of ternary pseudo-splines to define appropriate grid transfer operators.

\subsection{Binary primal pseudo-splines} \label{sec:examples_binary_ps}

We start our discussion by introducing the family of binary primal pseudo-spline schemes.

\begin{definition}[\cite{Dau_Han_Ron_Shen_pseudo}] \label{d:symbol_pseudo}
For integers $J \geq 1$ and $L=0,\dots,N-1$,  the \emph{binary primal pseudo-spline scheme} 
$S_{\mathbf{p}_{J,L}}$ of order $(J,L)$ is given by its symbol
\begin{equation} \label{eq:binary_primal_pseudo_spline_symbol}
p_{J,L}(z) = 2 \, \sigma^J (z) \, q_{J,L}(z), \qquad q_{J,L}(z) = \sum_{k=0}^L \binom{J-1+k}{k} \, \delta^k (z), \qquad z \in \CC \setminus \set{0},
\end{equation}
where \[
\sigma(z) = \frac{(1+z)^2}{4z} \qquad \text{and} \qquad \delta(z) = - \frac{(1-z)^2}{4z}. \]
\end{definition}
These pseudo-spline schemes range from B-splines to Dubuc-Deslauries schemes. When $L=0$ the symbol in \eqref{eq:binary_primal_pseudo_spline_symbol} is the symbol of the B-spline subdivision scheme of degree $2J-1$
and, when $L=J-1$, one gets the symbol of  the $(2J)$-point Dubuc-Deslauries interpolatory subdivision scheme. For more details on binary pseudo-splines see \cite{Dau_Han_Ron_Shen_pseudo, Dong_Shen, Dong_Shen_2, Floater_Muntingh}.

Next, we give several examples of grid transfer operators derived from symbols of binary primal pseudo-splines of order $(J,0)$, 
namely B-splines of degree $2J-1$. The symbols $p_{1,0}$ and $p_{2,0}$ have already been used in multigrid literature \cite{NLL2010,Serra_Tablino} as well the classical cubic interpolation $p_{2,1}$ (see \cite{mgm-book}).

\begin{example} \label{ex:binary_ps1}
For $J \geq 1$ and $L=0$, we have $q_{J,0}(z) \equiv 1$. Thus, from \eqref{eq:binary_primal_pseudo_spline_symbol}, we get
\[
p_{J,0} (z) = 2 \,  \left ( \frac{(1+z)^2}{4z} \right )^J, \qquad z \in \CC \setminus \set{0}.
\]
Set $z=e^{-ix}$, $x \in \RR$. Then the symbols $p_{J,0}$ become trigonometric polynomials 
\[
p_{J,0} (x) = 2 \, \left ( \frac{1 + \cos x}{2} \right )^{J}, \qquad x \in [0,2\pi),
\]
that are used to define grid transfer operators  in \eqref{eq:vcycle_notation}. For readers convenience, we also present the corresponding 
masks. For $J=1,2,3$, they are given by
\[ \begin{split}
\mathbf{p}_{1,0} = \frac12 &\begin{Bmatrix} 1 & 2 & 1 \end{Bmatrix}, \qquad \mathbf{p}_{2,0} = \frac18 \begin{Bmatrix} 1 & 4 & 6 & 4 & 1 \end{Bmatrix}, \\
&\mathbf{p}_{3,0} = \frac{1}{32} \begin{Bmatrix} 1 & 6 & 15 & 20 & 15 & 6 & 1 \end{Bmatrix}.
\end{split}
\]
Note that we use the corresponding grid transfer operators for our numerical examples in Tables \ref{table:binary_res} and \ref{table:binary_res_iga}.
\end{example}

Less known are grid transfer operators which we derive from symbols in \eqref{eq:binary_primal_pseudo_spline_symbol} for $L \not=0$.

\begin{example}  \label{ex:binary_ps2} Let $J =2$ and $L =1$, or $J=3$ and $L=1,2$. Then, from \eqref{eq:binary_primal_pseudo_spline_symbol}, using
standard trigonometric identities, we get
\[
\begin{split}
p_{2,1} (x) &= \frac{1}{16} \, \bigl (16 + 18 \cos x - 2 \cos (3x) \bigr ),  \\
p_{3,1} (x) &= \frac{1}{128} \, \bigl (110 + 144 \cos x + 24\cos (2x) -16 \cos (3x) -6 \cos (4x) \bigr ), \\
p_{3,2} (x) &= \frac{1}{256} \, \bigl (256 + 300 \cos x -50  \cos (3x) + 6 \cos(5x) \bigr ), \quad x \in [0, 2\pi).
\end{split}
\]
The corresponding masks are
\[
\begin{split}
\mathbf{p}_{2,1} &= \frac{1}{16}  \begin{Bmatrix} -1 & 0 & 9 & 16 & 9 & 0 & -1 \end{Bmatrix}, \\
\mathbf{p}_{3,1} &= \frac{1}{128} \begin{Bmatrix} -3 & -8 & 12 & 72 & 110 & 72 & 12 & -8 & -3 \end{Bmatrix}, \\
\mathbf{p}_{3,2} &= \frac{1}{256} \begin{Bmatrix} 3 & 0 & -25 & 0 & 150 & 256 & 150 & 0 & -25 & 0 & 3 \end{Bmatrix}.
\end{split}
\]
Note that the corresponding grid transfer operators also appear in Tables \ref{table:binary_res} and \ref{table:binary_res_iga}.\\
The symbols of the proposed grid transfer operators are plotted in Figure \ref{fig:pro} (a) for the reference interval $[0, \pi]$.
\end{example}

\begin{figure}
	\centering
	\begin{subfigure}{0.49\textwidth}
		\includegraphics[width=\textwidth]{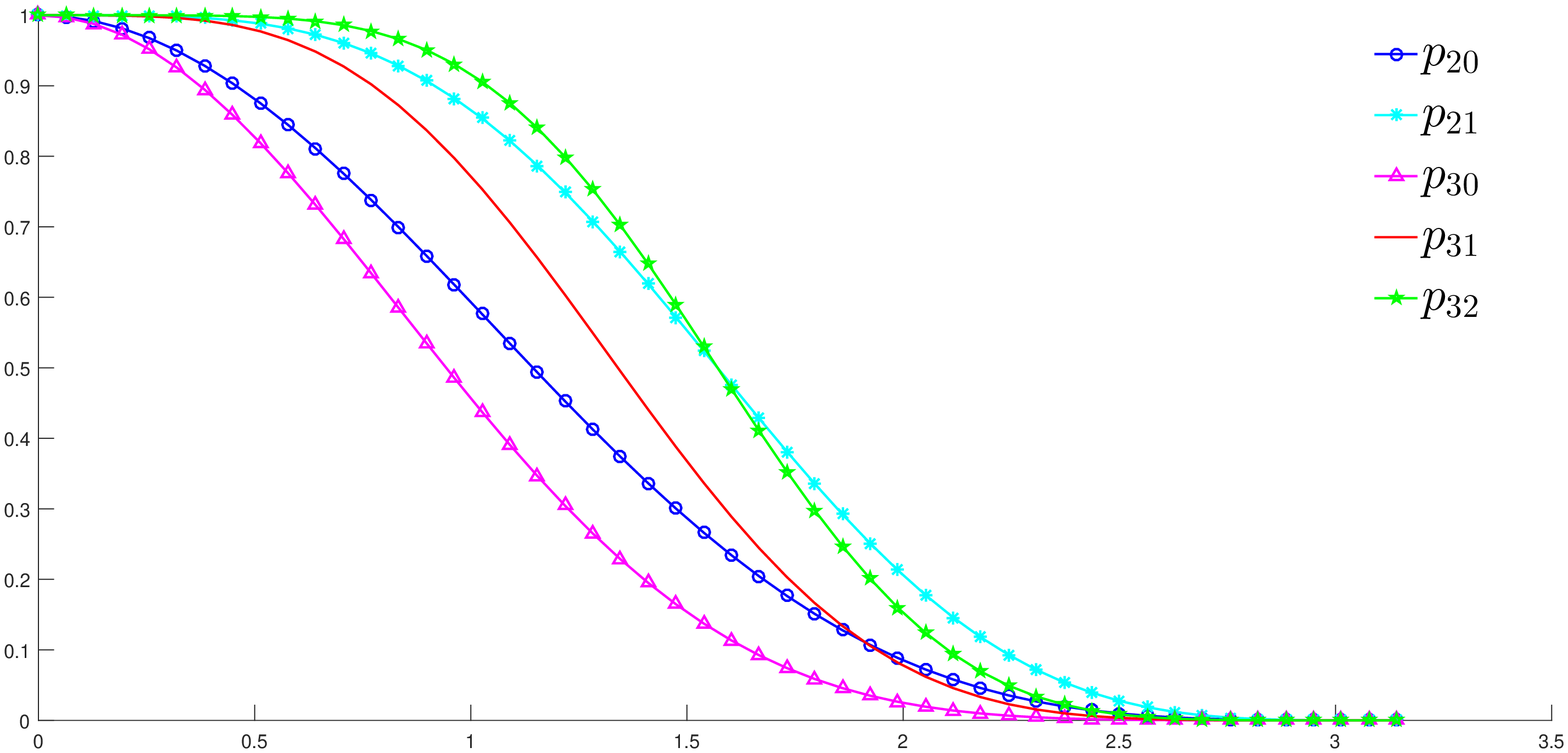}
		\caption{Binary pseudo-splines}
	\end{subfigure}
	\begin{subfigure}{0.49\textwidth}
		\includegraphics[width=\textwidth]{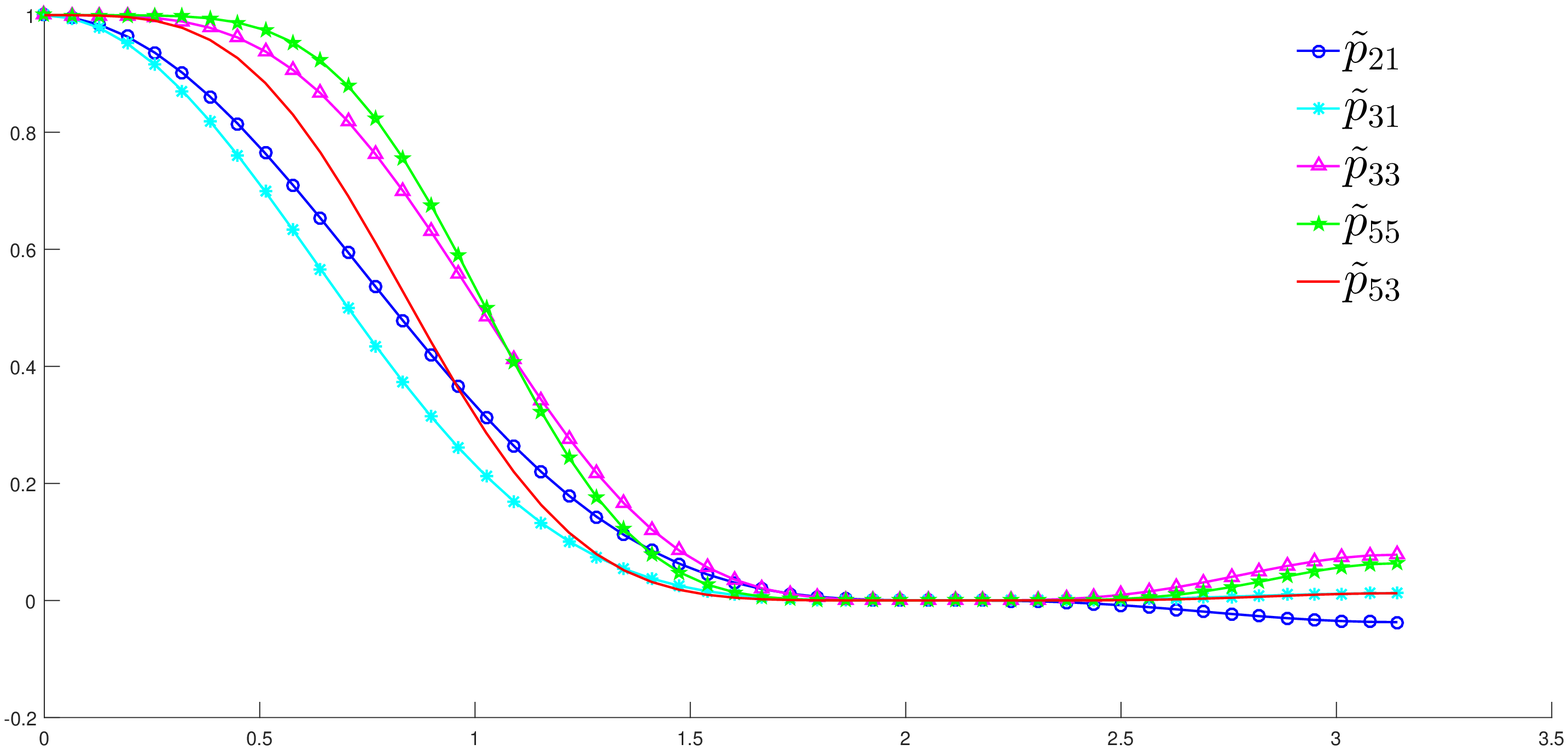}
		\caption{Ternary pseudo-splines}
	\end{subfigure}
	\caption{Symbols of the grid transfer operators defined in (a) by primal binary pseudo-splines and in (b) by primal ternary pseudo-splines in the interval $[0,\pi]$.}
	\label{fig:pro}
\end{figure}

The justification that primal pseudo-spline symbols define good grid transfer operators is given by Theorem \ref{t:subd_vcycle}. The convergence of primal pseudo-splines has been proved by Dong and Shen in \cite{Dong_Shen_2}.
The special structure of the symbols in \eqref{eq:binary_primal_pseudo_spline_symbol}, i.e. the presence of the factors $(1+z)$, implies that the corresponding schemes of order $(J,L)$ generate polynomials up to degree $2J-1$ 
for every $J \geq 1$, $L=0, \dots, J-1$. Thus, $(i)$ of Theorem \ref{t:subd_vcycle} is satisfied. Therefore, it is left to show that the corresponding
basic limit functions are $\ell^{\infty}$-stable. In \cite{Dong_Shen}, the authors addressed this issue. We present an alternative proof of $\ell^{\infty}$-stability of primal pseudo splines for completeness.
To do that, we first recall that in proof of Theorem 2 in \cite{Floater_Muntingh}, the authors showed the following.

\begin{lemma} \label{l:symbol_pseudo_stability}
Let $S_{\mathbf{p}}$ be a convergent subdivision scheme with associated symbol
\[
 p(z) = 2 \left ( \frac{1+z}{2} \right )^{r} \, z^{- \lfloor r/2 \rfloor } \, q(z), \qquad r \geq 1, \qquad z \in \CC \setminus \set{0}.
\]
If $q(e^{-\mathrm{i} x }) > 0$ for all $x \in \RR$, then the basic limit function $\phi$ of $S_{\mathbf{p}}$
is $\ell^{\infty}$-stable.
\end{lemma}
The result of Lemma \ref{l:symbol_pseudo_stability}, is used in the proof of Proposition \ref{p:pseudo_stable}.

\begin{proposition} \label{p:pseudo_stable}
Let $J\geq 1$, $L=0, \dots, J-1$. The basic limit function $\phi$ of $S_{\mathbf{p}_{J,L}}$  is $\ell^{\infty}$-stable.
\end{proposition}

\begin{proof}
By Definition \ref{d:symbol_pseudo}, for $J \geq 1$ and $L=0,\dots,J-1$, the symbols $p_{J,L}$ of primal pseudo-spline schemes 
$S_{\mathbf{p}_{J,L}}$ of order $(J,L)$ are of the form required by Lemma \ref{l:symbol_pseudo_stability}, with
\[
q(z) := q_{J,L} (z) =  \sum_{k=0}^L \binom{J-1+k}{k} \, \left ( - \frac{(1-z)^2}{4z} \right )^k, \qquad z \in \CC \setminus \set{0},
\]
i.e.
\[
 q(e^{-\mathrm{i} x }) = \sum_{k=0}^L \binom{J-1+k}{k} \, \sin^{2k} \left ( \frac{x}{2} \right ) > 0, \qquad \forall \, x \in \RR.
\]
\end{proof}
Thus,  $(ii)$ of Theorem \ref{t:subd_vcycle} is also satisfied and it implies
the following result.

\begin{proposition} \label{p:binary_pseudo_vcycle}
Let $f$ be a real trigonometric polynomial such that $f(x) > 0$, $x \in (0,2\pi)$, and $D^\mu f(0) = 0$, $\mu=0,\ldots,m-1$, $D^m f(0) \neq 0$. The corresponding grid transfer operator derived from the symbol
$p_{\lceil m/2 \rceil,L}$, $L \in \set{0, \ldots, \lceil m/2 \rceil-1}$,
 satisfies the approximation property \eqref{eq:approx_property_vcycle}.
\end{proposition}

\subsection{Ternary primal pseudo-splines} \label{sec:examples_ternary_ps}

\vspace{0.3cm} We show in section \ref{sec:num_examples}, in the case of PDE discretizations via isogeometric approach with high order B-splines, that the grid transfer operators derived from the binary primal pseudo-spline schemes
lead to computationally expensive multigrid methods. On the contrary,
if we use the ternary primal pseudo-spline schemes, the number of multigrid iterations decreases drastically. 

The recursive definition of ternary pseudo-splines was introduced in \cite{CH2011}. The explicit form of some of
those ternary pseudo-splines is due to personal communication with G. Muntingh.  

\begin{definition} \label{d:symbol_pseudo_ternary}
 Let $J \geq 1$ and $L=2 L' + 1, \, 1 \leq L \leq J$. The symbol $\tilde{p}_{J,L}$ of the \emph{ternary primal pseudo-spline scheme} of order $(J,L)$ is given by 
 \begin{equation} \label{eq:ternary_primal_pseudo_spline_symbol}
 \tilde{p}_{J,L} (z) = 3 \, \tilde{\sigma}^{J+1} (z) \, \tilde{q}_{J,L}(z), \qquad \tilde{q}_{J,L}(z) = \sum_{k=0}^{L'} \binom{J+k}{k} \, \tilde{\delta}^k (z), \qquad z \in \CC \setminus \set{0}, 
\end{equation}
where 
\[
 \tilde{\sigma}(z) = \frac{1+z+z^2}{3z} \qquad \text{and} \qquad \tilde{\delta}(z) = - \frac{(1-z)^2}{3z}. 
\]
\end{definition}

Similarly to the binary case, when $L=1$, the polynomial $\tilde{p}_{J,L}$ is the symbol of the ternary B-spline subdivision scheme 
of degree $J$ and, when $L=J$, $J$ odd, one gets the symbol of the ternary $(J+1)$-point Dubuc-Deslauries interpolatory subdivision scheme.

Next, we show how to derive grid transfer operators from symbols of some ternary primal pseudo-spline schemes.

\begin{example} \label{ex:ternary_ps} Let $J \geq 1$ and $L=1$. Then $\tilde{q}_{J,L}(z) \equiv 1$. From \eqref{eq:ternary_primal_pseudo_spline_symbol}, we obtain
the following symbols of primal pseudo-splines of order $(J,1)$, i.e symbols of the ternary B-splines of degree $J$, 
\[
\tilde{p}_{J,1}(z) = 3 \, \left( \frac{1+z+z^2}{3z} \right )^{J+1}, \qquad z \in \CC \setminus \set{0}.
\]
Set $z=e^{-ix}$, $x \in \RR$. Using simple trigonometric identities, we get the trigonometric polynomials
\[
\tilde{p}_{J,1} (x) =  3 \, \left( \frac{1 + 2 \cos x}{3} \right )^{J+1}, \qquad x \in [0,2\pi).
\]
The corresponding masks for linear ($J=1$), quadratic ($J=2$) and cubic ($J=3$) ternary B-splines are
\begin{gather*}
\mathbf{\tilde{p}}_{1,1} = \frac13 \begin{Bmatrix} 1 & 2 & 3 & 2 & 1 \end{Bmatrix}, \qquad \mathbf{\tilde{p}}_{2,1} = \frac19 \begin{Bmatrix} 1 & 3 & 6 & 7 & 6 & 3 & 1 \end{Bmatrix}, \\
\mathbf{\tilde{p}}_{3,1} = \frac{1}{27} \begin{Bmatrix} 1 & 4 & 10 & 16 & 19 & 16 & 10 & 4 & 1 \end{Bmatrix}.
\end{gather*}
Note that we use the corresponding grid transfer operators to obtain results  in Tables \ref{table:ternary_res} and \ref{table:ternary_res_iga}.
\end{example}

Further examples are obtained for odd $J > 1$ in the following example. These correspond to the ternary $(J+1)$-point Dubuc-Deslauries interpolatory subdivision schemes.

\begin{example} \label{ex:ternary_ps1} Let $J=3,5$ and $L=J$. From \eqref{eq:ternary_primal_pseudo_spline_symbol}, we derive the trigonometric polynomials
\[
\begin{split}
\tilde{p}_{3,3} (x) = \frac{1}{81} \bigl (&81 + 120 \cos x + 60 \cos (2x) - 10 \cos (4x) - 8 \cos (5 x) \bigr ), \\
\tilde{p}_{5,5} (x) = \frac{1}{729} \bigl (&729 + 1120 \cos x + 560 \cos (2x) - 140 \cos (4x) - 112 \cos (5x) + \\
&16 \cos (7x) + 14 \cos (8x) \bigr ),\qquad x \in [0, 2\pi).
\end{split}
\]
The corresponding masks are
\[
\begin{split}
\mathbf{\tilde{p}}_{3,3} = \frac{1}{81} &\begin{Bmatrix} -4 & -5 & 0 & 30 & 60 & 81 & 60 & 30 & 0 & -5 & -4 \end{Bmatrix}, \\
\mathbf{\tilde{p}}_{5,5} = \frac{1}{729} &\begin{Bmatrix} 7 & 8 & 0 & -56 & -70 & 0 & 280 & 560 & 729 \\
& 560 & 280 & 0 & -70 & -56 & 0 & 8 & 7 \end{Bmatrix}.
\end{split}
\]
Note that the corresponding grid transfer operators are also used in Tables \ref{table:ternary_res} and \ref{table:ternary_res_iga}.\\
Figure \ref{fig:pro} (b) depicts the symbols of the grid transfer operators in the previous two examples.
\end{example}

We use Theorem \ref{t:cohen_simple_vcycle} to show that ternary pseudo-splines lead to appropriate grid transfer operators. Note that we could also use Theorem \ref{t:cohen_simple_vcycle} in the binary case. 
To check the assumptions of Theorem \ref{t:cohen_simple_vcycle}, we need the following auxiliary lemma.

\begin{lemma} \label{lemma:Cohen_ternary_ps}
Let $J \geq 1$, $L=2 L' + 1, \, 1 \leq L \leq J$. The symbols $\tilde{p}_{J,L}$ of the ternary primal pseudo spline scheme of order $(J,L)$ satisfy
\begin{equation} \label{eq:cohen_ternary}
\abs*{\tilde{p}_{J,L} \left( e^{- \mathrm{i} x } \right)} > 0, \qquad \forall \, x \in \left [-\frac{\pi}{3}, \frac{\pi}{3} \right ].
\end{equation}
\end{lemma}

\begin{proof}
Define
\[
\tilde{P}(x) := \tilde{p}_{J,L} \left ( e^{- \mathrm{i} x} \right ) =  3 \tilde{B}_{J+1}(x) \, \tilde{Q} (x), \qquad x \in [0,2\pi], \]
where
\[
\tilde{B}_{J+1}(x) = \left ( \frac{1 + e^{- \mathrm{i} x} + e^{- 2 \mathrm{i} x}}{3 e^{- \mathrm{i} x} } \right )^{J+1}
\]
and
\[
\tilde{Q} (x) = \sum_{k=0}^{L'} \binom{J+k}{k} \, \left ( \frac43 \sin^2 \left ( \frac{x}{2} \right ) \right )^k.
\]
Note that $\tilde{B}_{J+1}(x)$ vanishes only at $-\frac{2 \pi}{3}$ and $\frac{2 \pi}{3}$. Thus, to check condition \eqref{eq:cohen_ternary}, it suffices to show that $\abs*{\tilde{Q}(x)} > 0$ for all
$x \in \left [-\frac{\pi}{3}, \frac{\pi}{3} \right ]$. The latter holds due to
\[
\tilde{Q} (x) = \sum_{k=0}^{L'} \binom{J+k}{k} \, \left ( \frac43 \sin^2 \left ( \frac{x}{2} \right ) \right )^k > 0, \quad \forall x \in \RR.
\]
\end{proof}
The presence of the factor $(1+z+z^2)^{J+1}$ in $\tilde{p}_{J,L}$ shows that the ternary pseudo spline-schemes of order $(J,L)$ generate 
polynomials up to degree $J$. This, together with $\tilde{p}_{J,L} (1) = 3$ by definition, implies that $\tilde{p}_{J,L}$ satisfies zero conditions of order
$J+1$. Thus, Theorem \ref{t:cohen_simple_vcycle} implies the following Proposition.

\begin{proposition} \label{p:ternary_pseudo_vcycle}
Let $f$ be a real trigonometric polynomial such that $f(x) > 0$, $x \in (0,2\pi)$, and $D^\mu f(0) = 0$, $\mu=0,\ldots,m-1$, $D^m f(0) \neq 0$.
The grid transfer operator derived from the symbol $\tilde{p}_{m-1,L}$, $L=2L' + 1 \in \set{1, \ldots, m-1}$, satisfies the approximation property \eqref{eq:approx_property_vcycle}.
\end{proposition}

\section{Numerical examples} \label{sec:num_examples}

In this section, we illustrate the results of Propositions \ref{p:binary_pseudo_vcycle} and \ref{p:ternary_pseudo_vcycle} on several examples.
In subsection \ref{sec:examples_biharmonic}, we consider several linear systems $A_n \mathbf{x} = \mathbf{b}_n$, $n = g^k$, $k \in \NN$, derived via finite difference discretization
from biharmonic elliptic PDE problem with homogeneous Dirichlet boundary conditions.  In subsection \ref{sec:examples_laplacian}, we consider
linear systems  $A_n \mathbf{x} = \mathbf{b}_n$ derived via isogeometric approach from Laplacian problem with homogeneous Dirichlet boundary conditions.
In both cases the system matrices are symmetric and positive definite.  The choice of the boundary conditions makes $A_n$ Toeplitz
and dictates the change in the definition of the grid transfer operators in \eqref{eq:vcycle_notation}. Namely,
let $\ell \in \NN$, $1 \leq \ell \leq k-1$ and  define
\begin{equation} \label{eq:vcycle_notation_toeplitz}
n_j = g^{k-j} - 1,  \qquad P_{n_j}(p) = T_{n_j} (p) \, \bar{Z}_{n_j}^T \in \CC^{n_j \times n_{j+1}}, \quad
j=0, \dots, \ell.
\end{equation}
The downsampling matrix $\bar{Z}_{n_j}$ is defined by
\[
\bar{Z}_{n_j} = \quad \begin{bmatrix}
0_{g-1} &1 & 0_{g-1} \\
& & & 1 & 0_{g-1} \\
& & & & & \ddots & \\
& & & & & & 1 & 0_{g-1}
\end{bmatrix} \quad \in \CC^{n_{j+1} \times n_j}.
\]
As pre- and post-smoother, we use one step of Gauss-Seidel method. In the binary case, we solve the coarse grid system 
exactly when the dimension of the coarse grid is $N=2^2-1=3$ and, in the ternary case, when $N=3^2-1=8$. The zero vector is used as the initial guess and 
the stopping criterion is $\norma{\mathbf{r}_s}_2 / \norma{\mathbf{r}_0}_2 < 10^{-7}$, where $\mathbf{r}_s$ is the residual vector after $s$ iterations and $10^{-7}$ is 
the given tolerance. 

\subsection{Biharmonic elliptic PDE} \label{sec:examples_biharmonic}

The first example we present arises from the discretization of a biharmonic elliptic PDE problem with homogeneous Dirichlet boundary conditions. For the discretization, 
we use finite differences of order 4. It leads to the linear systems $A_n \mathbf{x} = \mathbf{b}_n$, where $A_n = T_n (f)$ is the Toeplitz matrix with the symbol 
\[
f(x) = (2 - 2 \cos x)^2, \qquad x \in [0,2\pi).
\]
Note that $f$ has a quadruple zero at $x=0$. Thus, by Propositions \ref{p:binary_pseudo_vcycle} and \ref{p:ternary_pseudo_vcycle} with $m=4$,
the binary pseudo-spline symbols from Example  \ref{ex:binary_ps1} (with $J \ge 2$) and the ternary pseudo-spline symbols from Example \ref{ex:ternary_ps}
(with $J=3$) and all symbols from Examples \ref{ex:binary_ps1} and \ref{ex:ternary_ps1} can be used to define the corresponding grid transfer operators. 
To define $\mathbf{b}_n$, we choose $\mathbf{x}=(x_1, \ldots, x_n)$, $x_j = j/n, \, j = 1, \dots , n$ and set $\mathbf{b}_n := A_n \mathbf{x}$.

Tables \ref{table:binary_res} and \ref{table:ternary_res} show how the number of iterations and convergence rates for the V-cycle change with 
increasing dimension $n$.

\begin{table} 
\centering
\footnotesize
\begin{tabular}{l@{\qquad\;\;}cc@{\qquad\;\;}cc@{\qquad\;\;}cc@{\qquad\;\;}c}
\toprule
{Subdivision} & \multicolumn{2}{c@{\qquad\;\;}}{$n=2^{10}-1$}  & \multicolumn{2}{c@{\qquad\;\;}}{$n=2^{11}-1$}  & \multicolumn{2}{c@{\qquad\;\;}}{$n=2^{12}-1$} & {gen.} \\
{scheme} & {iter} & {conv. rate} & {iter}  & {conv. rate} & {iter} & {conv. rate} & {deg.} \\
\toprule
$p_{1,0}$ (Linear Bspline) & 617 & 0.9742 & 744 & 0.9785 & 801  & 0.9800 & 1 \\
$p_{2,0}$ (Cubic Bspline) & 40 &  0.6647 & 43 &  0.6846 & 45 &  0.6979 & 3 \\
$p_{2,1}$ (Interp. 4 point) & 19 & 0.4275 & 23 & 0.4937 & 26 & 0.5351 & 3 \\
$p_{3,0}$ (Quintic Bspline) & 30 & 0.5784 & 35 &  0.6285 & 41 & 0.6741 & 5\\
$p_{3,1}$ & 19 &  0.4258 & 22 & 0.4748 & 24  & 0.5063 & 5 \\
$p_{3,2}$ (Interp. 6 point) & 13 &  0.2798 & 13 & 0.2879 & 14 & 0.3080 & 5 \\
\bottomrule
\end{tabular}
\caption{Binary subdivision schemes for biharmonic problem}
\label{table:binary_res}
\end{table}

\begin{table} 
\centering
\footnotesize
\begin{tabular}{l@{\qquad\;\;}cc@{\qquad\;\;}cc@{\qquad\;\;}cc@{\qquad\;\;}c}
\toprule
{Subdivision} & \multicolumn{2}{c@{\qquad\;\;}}{$n=3^{6}-1$} & \multicolumn{2}{c@{\qquad\;\;}}{$n=3^{7}-1$} & \multicolumn{2}{c@{\qquad\;\;}}{$n=3^{8}-1$} & {gen.}\\
{scheme} & {iter} & {conv. rate} & {iter} & {conv. rate} & {iter} & {conv. rate} & {deg.} \\
\toprule
$\tilde{p}_{1,1}$ (Linear Bspline) & 462 & 0.9656 & 864 & 0.9815 & 1057 & 0.9841 & 1 \\
$\tilde{p}_{2,1}$ (Quadratic Bspline) & 72  & 0.7990 & 63 & 0.7742 & 50 & 0.7217 & 2 \\
$\tilde{p}_{3,1}$ (Cubic Bspline) & 67  & 0.7858 & 80 & 0.8167 & 87 & 0.8308 & 3 \\
$\tilde{p}_{3,3}$ (Interp. 4-point) & 46 & 0.7017 & 47 & 0.7090 & 53 & 0.7368 & 3 \\
$\tilde{p}_{5,3}$  & 30 & 0.5824 & 31& 0.5878 & 30 & 0.5814 & 5 \\
$\tilde{p}_{5,5}$ (Interp. 6-point) & 39  & 0.6594 & 39  & 0.6604 & 40 & 0.6644 & 5 \\
\bottomrule
\end{tabular}
\caption{Ternary subdivision schemes for biharmonic problem}
\label{table:ternary_res}
\end{table}

Tables \ref{table:binary_res} and \ref{table:ternary_res} illustrate the importance of the polynomial generation property (zero conditions) that, by Theorems~\ref{t:subd_vcycle} and 
\ref{t:cohen_simple_vcycle}, ensures the 
correct choice of the grid transfer operator. The subdivision schemes with the symbols $p_{1,0}$, $\tilde{p}_{1,1}$  generate polynomials of degree $1$. 
The lack of the appropriate degree of polynomial generation leads to dramatic increase of
the number of iterations. 
For ternary schemes $\tilde{p}_{2,1}$ generate polynomials of degree $2$ and so it does not satisfy the assumptions of Theorem~\ref{t:cohen_simple_vcycle}. Nevertheless, such conditions are only sufficient and they could be further relaxed (see e.g. \cite{Serra_Tablino}). Moreover, the quadratic B-splines are very effective as grid transfer operator for ternary methods as shown also in the next example.

We observe that the number of iterations necessary for convergence of the V-cycle is larger in the ternary case (see Table \ref{table:ternary_res}) than 
in the binary case (see Table \ref{table:binary_res}). This happens, since, at each Coarse Grid Correction step, we downsample the data with the factor $g$ and
the larger is $g$ the more information we lose. Thus, the number of iterations required for convergence is larger for $g=3$. Nevertheless, in our tests, the CPU time 
is comparable in both cases, since the length of the V-cycle iteration is shorter in the ternary case ($g=3$).

\subsection{Laplacian problem} \label{sec:examples_laplacian}

In the second example we consider the Laplacian problem with homogeneous Dirichlet boundary conditions
\[
\begin{cases}
-u''(x) = h(x) & x \in [0,1], \\
u(0) = u(1) = 0. &
\end{cases}
\]
We consider the isogeometric approach with collocation by splines for the discretization of the above problem, see~ \cite{Don_Gar_Mann_Serra_Spel2}.  
We fix the integers $\nu, \, \mu > 0$ and define the spline space
\[
{\cal W} = \Set{s \in C^{\mu-1}([0,1]) \, : \, s_{ \big |_{ \left [ \frac{j}{\nu}, \frac{j+1}{\nu} \right )} }  \in \Pi_{\mu}, \, j = 0, \dots, \nu-1, \, s(0) = s(1) = 0},
\]
the finite dimensional approximation space of dimension $n = \text{dim }{\cal W} = \nu+\mu-2$. As a basis for ${\cal W}$, one chooses the 
B-splines $B_{j}^{[\mu]} \colon [0,1] \to \RR$, $j=2, \ldots, \nu+\mu-1$ of degree $\mu$ as explained in \cite{DeBoor_spline}. These are defined over the uniform knot sequence of length $\nu + 2 \mu +1$
\[
t_1 = \dots = t_{\mu+1} = 0 < t_{\mu+2} < \dots < t_{\mu + \nu} < 1 = t_{\nu+\mu + 1} = \dots = t_{\nu + 2 \mu +1},
\]
where
\[
t_{\mu + j + 1} = \frac{j}{\nu}, \quad j = 1, \dots, \nu-1,
\]
and the extreme knots have multiplicity $\mu+1$. We recall that the B-splines $B_{j}^{[\mu]} \colon [0,1] \to \RR$  are  defined recursively by
\[
B_j^{[0]} (x) = \begin{cases}
1 & x \in [t_j, t_{j+1}), \\
0 & \text{otherwise},
\end{cases}
\quad j = 1, \dots, \nu + 2\mu,
\]
and
\[
B_j^{[m]} (x) = \frac{x - t_j}{t_{j+m} - t_j} \,  B_j^{[m-1]} (x) + \frac{t_{j+m+1} - x}{t_{j+m+1} - t_{j+1}} \, B_{j+1}^{[m-1]} (x), 
\]
$j = 1, \dots, (\nu + \mu) + \mu - m, \, m = 1, \dots, \mu,$, where we set the fractions with zero denominators to be equal to zero. Next, one defines the set of collocation points, the so-called Greville abscissae,
\[
\tau_j = \frac{t_{j+1} + t_{j+2} + \dots + t_{j+\mu}}{\mu}, \quad j = 2, \dots, \nu + \mu -1.
\]
This choice is crucial for the stability of the discrete problem, see \cite{Aur_DaVeiga_Hughes_Reali_Sang_IGA} for more details.
The solution $\mathbf{u}_{\cal W} \in {\cal W}$ of the interpolation problem
$$
 -u_{\cal W}''(\tau_j) = h(\tau_j), \quad j = 2, \dots, \nu + \mu -1,
$$
written in the Bspline basis of ${\cal W}$ leads to
\[
 A_n = \left [ \, - \left ( B_{\ell+1}^{[\mu]} \right )'' (\tau_{j+1}) \, \right ]_{j, \ell = 1}^{n} \in \CC^{n \times n}.
\]
For $\mu \geq 2$, it is possible to split the above matrix into $A_n = T_n (f^{[\mu]}) + R_n^{[\mu]}$, where $T_n (f^{{\mu}})$ is a Toeplitz matrix with symbol
\begin{equation} \label{eq:symbol_iga}
f^{[\mu]} (x) =  (2 - 2 \cos x) h^{[\mu]} (x), \quad h^{[\mu]} (x) = \sum_{\alpha \in \ZZ} \left ( \frac{2 \sin (x/2) + \alpha \pi}{x + 2 \alpha \pi} \right )^{\mu - 1},
\end{equation}
$x \in [0,2\pi)$, and $R_n^{[\mu]}$ is a low rank correction term, see \cite{Don_Gar_Mann_Serra_Spel}. The symbols for the grid transfer operators are chosen as
in Example~\ref{sec:examples_biharmonic}. 
To define $\mathbf{b}_n$, we choose the exact solution
\[
 \mathbf{x} = (x_1, \ldots, x_n)^T, \quad x_j =\sin \left ( 5 \frac{\pi (j-1)}{n-1} \right ) + \sin \left ( n \frac{\pi (j-1)}{n-1} \right ), \quad j = 1, \dots , n,
\]
and set $\mathbf{b}_n = A_n \mathbf{x}$.

\begin{table} \label{table_3}
\centering
\footnotesize
\begin{tabular}{l@{\qquad\;\;}cc@{\qquad\;\;}cc@{\qquad\;\;}cc@{\qquad\;\;}c}
\toprule
{Subdivision} & \multicolumn{2}{c@{\qquad\;\;}}{$\mu = 3$}  & \multicolumn{2}{c@{\qquad\;\;}}{$\mu = 10$}  & \multicolumn{2}{c@{\qquad\;\;}}{$\mu = 16$} & {gen.} \\
{scheme} & {iter} & {conv. rate} & {iter} & {conv. rate} & {iter} & {conv. rate} & {deg.} \\
\toprule
$p_{1,0}$ (Linear Bspline) & 8  & 0.1111 & 16 & 0.3360 & 126 & 0.8798 & 1 \\
$p_{2,0}$ (Cubic Bspline) & 8  & 0.1111 & 13 & 0.2757 & 126 & 0.8799 & 3 \\
$p_{2,1}$ (Interp. 4 point) & 8  & 0.1111 & 13 & 0.2758 & 126 & 0.8799 & 3 \\
$p_{3,0}$ (Quintic Bspline) & 8  & 0.1111 & 13  & 0.2758 & 126 & 0.8798 & 5\\
$p_{3,1}$ & 8  & 0.1111 & 13  & 0.2759 & 126 &  0.8798 & 5 \\
$p_{3,2}$ (Interp. 6 point) & 8  & 0.1111 & 13 & 0.2759 & 126 & 0.8798 & 5 \\
\bottomrule
\end{tabular}
\caption{Binary subdivision schemes for isogeometric Laplacian problem}
\label{table:binary_res_iga}
\end{table}

\begin{table} \label{table_4}
\centering
\footnotesize
\begin{tabular}{l@{\qquad\;\;}cc@{\qquad\;\;}cc@{\qquad\;\;}cc@{\qquad\;\;}c}
\toprule
{Subdivision} & \multicolumn{2}{c@{\qquad\;\;}}{$\mu = 3$}  & \multicolumn{2}{c@{\qquad\;\;}}{$\mu = 10$}  & \multicolumn{2}{c@{\qquad\;\;}}{$\mu = 16$} & {gen.} \\
{scheme} & {iter} & {conv. rate} & {iter} & {conv. rate} & {iter} & {conv. rate} & {deg.} \\
\toprule
$\tilde{p}_{1,1}$ (Linear Bspline) & 31  & 0.5910 & 25 & 0.5247 & 48  & 0.7078 & 1 \\
$\tilde{p}_{2,1}$ (Quadratic Bspline) & 30 & 0.5847 & 19  & 0.4271 & 49  & 0.7124 & 2 \\
$\tilde{p}_{3,1}$ (Cubic Bspline) & 29  & 0.5739 & 16  & 0.3617 & 49  & 0.7120 & 3 \\
$\tilde{p}_{3,3}$ (Interp. 4-point) & 30 & 0.5853 & 17& 0.3731& 49  & 0.7118 & 3 \\
$\tilde{p}_{5,3}$  & 28 & 0.643 & 16 & 0.358 & 49  & 0.7137 & 5 \\
$\tilde{p}_{5,5}$ (Interp. 6-point) & 30 & 0.5831 & 16 &0.3523 & 49 & 0.7120 & 5 \\
\bottomrule
\end{tabular}
\caption{Ternary subdivision schemes for isogeometric Laplacian problem}
\label{table:ternary_res_iga}
\end{table}

Tables \ref{table:binary_res_iga} and \ref{table:ternary_res_iga} show how the number of iterations and convergence rates for the V-cycle change with 
increasing $\mu$ and fixed $n$. The starting dimension of the linear systems are  $n=2^{9}-1$ and $n=3^{6}-1$ in the binary and the ternary cases, respectively.
For small $\mu$, the results in Tables \ref{table:binary_res_iga} and \ref{table:ternary_res_iga} mimic the ones from Example \ref{sec:examples_biharmonic}.
Note that, in this case, even the grid transfer operators defined from the subdivision symbols $p_{1,0}$, $\tilde{p}_{1,1}$ and $\tilde{p}_{2,1}$ behave well,
as the order of $f^{[\mu]}$ at zero is $m=2$ in this case. Thus, Propositions \ref{p:binary_pseudo_vcycle} and \ref{p:ternary_pseudo_vcycle} are also applicable for these symbols.
However, when $\mu$ increases, the results in the binary and ternary cases differ. This is the case, since the symbol $f^{[\mu]}$ in \eqref{eq:symbol_iga} has 
a numerical zero at $\pi$ whose order increases when $\mu$ increases, see Figure \ref{f:symbol_iga}. In fact, by \cite{Don_Gar_Mann_Serra_Spel}, $h^{[\mu]} (\pi)$ in \eqref{eq:symbol_iga} 
converges to $0$ exponentially when $\mu$ goes to infinity. The symbols $p_{N,L}$ also vanish at $\pi$ for $N \geq 1$ and $L=0, \dots, N-1$, which is the source of further ill-conditioning. Note that the ternary symbols $\tilde{p}_{N,L}$ do not vanish at $\pi$ and, hence, lead to more stable methods for increasing $\mu$. 
On the contrary, for small $\mu$, the ternary symbols are not at all a good choice for the definition of a grid transfer operator (compare Tables 3 and 4). 

\begin{figure}
\centering
\includegraphics[scale=0.25,trim={1cm 0.5cm 1cm 1cm},clip]{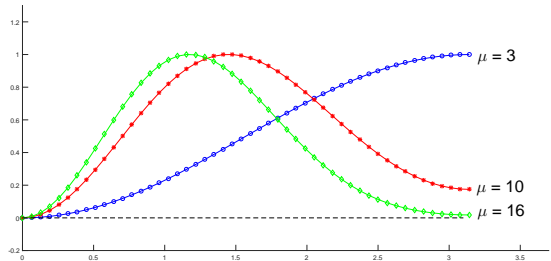}
\caption{Symbols $f^{[\mu]} / \norma{f^{[\mu]}}_{\infty}$ for $\mu \in \{3,10,16\}$ in $[0,\pi]$.}
\label{f:symbol_iga}
\end{figure}

\section{Conclusions}\label{sec:concl}
In this paper, we have shown that symbols of univariate subdivision schemes can be used to define powerful grid transfer operators in multigrid methods. Our analysis led to the definition of a whole class of new grid transfer operators.

In order to keep the presentation simple, in this paper, we discussed only one dimensional problems,
vertex centered discretizations and univariate primal subdivision schemes.  This allows for the first, transparent and straightforward exposition of the link between the symbol analysis for multigrid methods and symbols of subdivision schemes. Our results can be extended in many directions. In particular, the study of multivariate anisotropic problems  and multivariate subdivision schemes with general dilation matrices, or multigrid methods for linear
systems derived via cell centered discretizzations and dual subdivision schemes are of future interest.

\vspace{0.5cm} {Acknowledgments:} Maria Charina was sponsored by the Austrian Science Foundation (FWF) grant P28287-N35. Valentina Turati was sponsered by the OeAD' Austrian Office. 
This work was partially supported by Italian funds from MIUR-PRIN 2012 (grant 2012MTE38N).

\bibliographystyle{plain}
\bibliography{bibliography}

\end{document}